\nonstopmode
\pdfoutput=1
\documentclass[a4paper]{amsart}
\usepackage{amsmath, amssymb, amsthm, amscd}
\usepackage{todonotes}
\usepackage{graphicx}
\usepackage{bbm}
\usepackage[all]{xy}
\newdir{ >}{{}*!/-10pt/@{>}}
\newdir^{ (}{{}*!/-5pt/@^{(}}
\newdir_{ (}{{}*!/-5pt/@_{(}}
\def\cgaps#1{}
\def\Cgaps#1{}

\def\undersetbrace#1\to#2{\underbrace{#2}_{#1}}                                                          
\def\oversetbrace#1\to#2{\overbrace{#2}^{#1}}
\def\AMSunderset#1\to#2{\underset{#1}{#2}}
\def\AMSoverset#1\to#2{\overset{#1}{#2}}

\swapnumbers

\newtheorem*{prop*}{Proposition}

\newtheorem*{thm*}{Theorem}

\newtheorem*{lem*}{Lemma}

\newtheorem*{cor*}{Corollary}
\newtheorem*{remark*}{Remark}

\def\ign#1{}             
\def\o{\circ\,}
\def\X{\mathfrak X}
\def\al{\alpha}
\def\be{\beta}
\def\ga{\gamma}
\def\de{\delta}
\def\ep{\varepsilon}
\def\ze{\zeta}
\def\et{\eta}
\def\th{\theta}

\def\ka{\kappa}

\def\rh{\rho}

\def\ph{\varphi}

\def\ps{\psi}
\def\om{\omega}
\def\Ga{\Gamma}

\def\Om{\Omega}
\def\i{^{-1}}
\def\x{\times}
\def\p{\partial}
\let\on=\operatorname
\def\L{\mathcal L}

\def\AMSonly#1{}

\def\R{\mathbb R}

\def\Diff{{\on{Diff}}}
\def\Fl{{\on{Fl}}}

\def\A{{\mathcal{A}}}
\def\M{{\mathcal{M}}}
\def\AA{{\mathcal{A}_1}}
\let\wt=\widetilde

\title[The Sobolev $H^1$-metric on $\Diff(\mathbb R)$] 
{The homogeneous Sobolev  metric of order one on diffeomorphism groups on  the real line} 

\author{Martin Bauer, Martins Bruveris, Peter W. Michor}
 
\address{
Martin Bauer, Peter W.\ Michor: Fakult\"at f\"ur Mathematik,
Universit\"at Wien, 
Oskar-Morgenstern-Platz 1, A-1090 Wien, Austria.\newline\indent
Martins Bruveris: Institut de math\'ematiques, EPFL, CH-1015 Lausanne, Switzerland}
\email{bauer.martin@univie.ac.at}
\email{martins.bruveris@epfl.ch}
\email{peter.michor@univie.ac.at}

\date{{\today} } 

\thanks{MB was supported by `Fonds zur
F\"orderung der wissenschaftlichen                    
Forschung, Projekt P~24625'} 

\begin{document}
\begin{abstract}
In this article we study Sobolev metrics of order one on diffeomorphism groups on the real line.
We prove that the space $\on{Diff}_{1}(\R)$ equipped with the homogeneous Sobolev metric of order 
one is a flat space in the sense of Riemannian geometry, as it is isometric to an open subset of a 
mapping space equipped with the flat $L^2$-metric. Here $\on{Diff}_{1}(\R)$ denotes the extension 
of the group of all either compactly supported, rapidly decreasing, or 
$W^{\infty,1}$-diffeomorphisms, which allows for a shift toward infinity. 
Surprisingly, on the non-extended group the Levi-Civita connection does not exist.    
In particular, this result provides an analytic solution formula for the corresponding geodesic 
equation, the non-periodic Hunter--Saxton (HS) equation.  
In addition, we show that one can obtain a similar result for the two-component HS 
equation and discuss the case of the non-homogeneous Sobolev one metric, which is related to the Camassa--Holm equation. 
\end{abstract}

\maketitle 

\section{Introduction}

In recent decades it has been shown that various prominent partial differential equations (PDEs) arise as geodesic equations on certain infinite-dimensional manifolds. This phenomenon was first observed in the groundbreaking paper \cite{Arnold1966} for the incompressible Euler equation, which is the geodesic equation on the group of volume preserving diffeomorphisms with respect to the right-invariant $L^2$-metric. It was shown that this geometric approach could be extended to a whole variety of other PDEs used in hydrodynamics:  the  Camassa-Holm equation \cite{Holm1993,Kouranbaeva1999}, the  Constantin-Lax-Majda equation \cite{Wunsch2010a,Bauer2011c_preprint}
and the Korteweg--de Vries equation \cite{Ovsienko1987,Bauer2012c} to name but a few examples.

It was later realized that the geometric interpretation could be used to prove results about the behaviour of the PDEs. The first such result was \cite{Ebin1970}, where the researchers showed the local well-posedness of Euler equations. Similar techniques were then applied to other PDEs, which arise as geodesic equations, see e.g. \cite{Constantin2007, Constantin2003, Wunsch2010, GayBalmaz2009,Bauer2011b}.

The analysis in this paper is mainly concerned with the Hunter--Saxton (HS) equation.
For the periodic case it was shown in \cite{Misiolek2003} that the HS equation is the geodesic equation on 
the homogeneous space $\Diff(S^1)/S^1$ with respect to the homogeneous Sobolev $\dot H^1$-metric of order one.
Lenells used this geometric interpretation in \cite{Lenells2007,Lenells2008} to construct an analytic solution formula for the equation. 
In fact, he showed that the Riemannian manifold $\big(\Diff(S^1)/S^1,\dot{H}^1\big)$ was isometric to an open subset of an $L^2$-sphere in the space $C^\infty(S^1, \R)$ of periodic functions and thereby obtained an explicit formula for the corresponding geodesics on 
$\Diff(S^1)/S^1$. 

The aim of this article is twofold. First, we extend the results of \cite{Lenells2007} to groups of 
real-analytic and ultra-differentiable diffeomorphisms and show that the solutions of the 
HS equation are analytic in time and space. Second we consider the $\dot H^1$-metric and the 
HS equation on the real line. Our main result can be paraphrased as follows (see 
Sect.~\ref{section4}).

\begin{thm*}
The non-periodic HS equation is the geodesic equation on the space $\Diff_{\AA}(\R)$ with respect 
to the right-invariant $\dot H^1$-metric. Furthermore, the space 
$\big(\Diff_{\AA}(\R),\dot{H}^1\big)$ is isometric to an open subset in $\big(\A(\R),L^2\big)$ and is thus a flat space in the sense of Riemannian geometry.
\end{thm*}

Here $\A(\R)$ denotes one of the function spaces $C^{\infty}_c(\R)$, 
$\mathcal S(\R)$ or $W^{\infty,1}(\R)$ and $\Diff_{\AA}(\R)$ is an extension over the diffeomorphism group including 
shifts near $+\infty$ (Sect.~\ref{diffgroups}). 

The first surprising fact is that the normal subgroups $\Diff_\A(\mathbb R)$ do not admit the geodesic 
equation (or the Levi-Civita covariant derivative) for this right-invariant metric. 
The extended group $\Diff_\AA(\mathbb R)$ admits it but in a weaker sense than
realized in \cite{Arnold1966} and follow-up papers (Sect. \ref{adjoint}). We also sketch Arnold's 
curvature formula in this weaker setting. 

The second surprising fact is that $\on{Diff}_\AA(\R)$ with the $\dot H^1$-metric is a 
flat Riemannian manifold, as opposed to $\on{Diff}(S^1)$, which is a positively curved one.  

The main ingredient for the proof of this result is the $R$-map, which allows us to isometrically embed $\Diff_\AA(\mathbb R)$ 
with the right-invariant $\dot H^1$-metric as an open subset of a (flat) 
 pre-Hilbert space. This phenomenon has also been observed for the space $\on{Imm}(S^1,\R^2)/\R^2$ 
 of plane curves modulo translations (see. e.g., 
 \cite{Kurtek2010,Michor2008a,Bauer2012b_preprint}),
and on the semi-direct product space $\Diff(S^1)\ltimes C^\infty(S^1,\R)$
(the corresponding geodesic equation is the periodic two-component HS equation;
see \cite{Lenells2011_preprint}).

In the periodic case we extend the results to groups of real-analytic 
diffeomorphisms and ultra-differentiable diffeomorphisms of certain types and show that the HS equation has solutions which are real-analytic or ultra-differentiable if the initial diffeomorphism is.

In Sect. \ref{2hs} we apply the same techniques to treat the two-component HS equation on the real line. We discuss the existence of the geodesic equation and construct an isometry between the configuration space $\Diff_\A(\R)\ltimes \A(\R)$ and an open subset of a pre-Hilbert space. 

Finally, we generalize the constructions to the case of the right-invariant non-homogeneous $H^1$-metric on $\Diff(S^1)$, 
whose geodesic equation is the dispersion-free Camassa--Holm equation. In this case we define an $R$-map, whose image is a subspace of a pre-Hilbert space, no longer open.

\subsection*{Acknowledgments}
We thank Yury  Neretin for helpful discussions and comments regarding  the space $H^{\infty}(\R)$.

\section{Some Diffeomorphism Groups on the Real Line and the Circle}\label{diffgroups}

The group of all orientation-preserving diffeomorphisms $\Diff(\mathbb R)$ is not an open subset of 
$C^\infty(\mathbb R,\mathbb R)$ endowed with the compact $C^\infty$-topology and so 
it is not a smooth manifold with charts in the usual sense. One option is to consider it as a Lie group in the 
cartesian closed category of  Fr\"olicher spaces (see \cite[Sect. 23]{Michor1997}) with the structure 
induced by the injection 
$f\mapsto (f,f\i)\in C^\infty(\mathbb R,\mathbb R)\x C^\infty(\mathbb R,\mathbb R)$.
Alternatively, one can use the theory of smooth manifolds based on smooth curves instead of charts from 
\cite{Michor1984a}, \cite{Michor1984b}, which agrees with the usual theory up to Banach manifolds. In this paper we will restrict our attention to subgroups of the whole diffeomorphism group, which are smooth Fr\'echet manifolds.

Let us first briefly recall the definition of a regular Lie group in the sense of \cite{Kriegl1997}; see also \cite[Sect. 38.4]{Michor1997}. 
A  smooth Lie group $G$ with Lie algebra $\mathfrak g=T_eG$ 
is called regular if the following conditions hold:
\begin{itemize}
\item 
For each smooth curve 
$X\in C^{\infty}(\mathbb R,\mathfrak g)$ there exists a smooth curve 
$g\in C^{\infty}(\mathbb R,G)$ whose right logarithmic derivative is $X$, i.e.
\begin{equation}
\label{eq:regular}
\begin{cases} g(0) &= e \\
\p_t g(t) &= T_e(\mu^{g(t)})X(t) = X(t).g(t).
\end{cases} 
\end{equation}
The curve $g$, if it exists, is uniquely determined by its initial value $g(0)$.
\item
The map $\on{evol}^r_G(X)=g(1)$ where $g$ is the unique solution of \eqref{eq:regular}, considered as a map $\on{evol}^r_G: C^{\infty}(\mathbb R,\mathfrak g)\to G$ is $C^\infty$-smooth. 
\end{itemize}

\subsection{The Group $\Diff_{\mathcal B}(\mathbb R)$}
The `largest' regular Lie group in $\Diff(\mathbb R)$ with charts is the group of all diffeomorphisms 
$\ph = \on{Id}_{\mathbb R} + f$ with $f\in \mathcal B(\mathbb R)$ such that $f'>-1$.
\noindent
$\mathcal B(\mathbb R)$ is the space of functions which have all derivatives (separately) bounded.
It is a reflexive nuclear Fr\'echet space. 
\newline
{\it The space $C^\infty(\mathbb R,\mathcal{B}(\mathbb R))$ of smooth
curves $t\mapsto f(t,\cdot)$ in $\mathcal{B}(\mathbb R)$ consists of all functions 
$f\in C^\infty(\mathbb R^{2},\mathbb R)$ satisfying the following 
property:
\begin{enumerate}
\item[$\bullet$]
For all $k\in \mathbb N_{\ge0}$ and $n\in \mathbb N_{\ge0}$ the expression
$\p_t^{k}\p^n_xf(t,x)$  is uniformly bounded in $x\in \mathbb R$, locally
in $t$. 
\end{enumerate} 
}

We can specify other regular Lie groups by 
requiring that $g$ lies in certain spaces of smooth functions. 
Now we will discuss these spaces, describe the smooth curves in them, and describe the 
corresponding groups, specializing the results from \cite{Michor2012b_preprint}, where most of these groups are 
treated on $\mathbb R^n$ in full detail.

\subsection{Groups Related to  $\Diff_{c}(\mathbb R)$}\label{diffc} 
The reflexive nuclear (LF) space 
$C^\infty_c(\mathbb R)$ of smooth functions with compact support leads to the well-known regular Lie 
group $\Diff_{c}(\mathbb R)$, see \cite[Sect.~43.1]{Michor1997}.

We will now define an extension of this group which will play a major role in subsequent parts of this article.

\noindent
Define $C^\infty_{c,2}(\mathbb R)=\{f: f'\in C^\infty_c(\mathbb R)\}$ as the space of antiderivatives 
of smooth functions with compact support. It is a reflexive nuclear (LF) space.
We also define the space $C^\infty_{c,1}(\R) = \left\{ f\in C^\infty_{c,2}(\R) \;:\; f(-\infty)=0\right\}$ of antiderivatives of the form $x \mapsto \int_{-\infty}^x g\;dy$ with $g \in C^\infty_{c}(\R)$.
\newline
 {\it    $\Diff_{c,2}(\mathbb R)=\bigl\{\ph=\on{Id}+f: f\in C^\infty_{c,2}(\mathbb R), 
 f'>-1\bigr\}$ is the corresponding group.}  \newline
Define the two functionals $\on{Shift}_\ell, \on{Shift}_r : \on{Diff}_{c,2}(\R) \to \R$ 
by $$\on{Shift}_\ell(\ph) = \on{ev}_{-\infty}(f)=\lim_{x\to-\infty} f(x),\quad\on{Shift}_r(\ph) = \on{ev}_{\infty}(f)=\lim_{x\to\infty} f(x)$$ 
for $\ph(x) = x + f(x)$.
Then the short exact sequence of smooth homomorphisms of Lie groups 
$$\xymatrix{
\Diff_{c}(\mathbb R)\quad \ar@{>->}[r] & 
\Diff_{c,2}(\mathbb R) 
\ar@{->>}[rrr]^{(\on{Shift}_{\ell},\on{Shift}_r)} &&& (\mathbb R^2,+)
}
$$
describes a semi-direct product, where a smooth homomorphic section $s: \mathbb R^2\to \Diff_{c,2}(\mathbb R)$ is 
given by the composition of flows $s(a,b)= \Fl^{X_{\ell}}_a\o \Fl^{X_r}_b$ for the vector fields 
$X_{\ell} = f_{\ell}\p_x$, $X_r = f_r\p_x$ with $[X_{\ell},X_r]=0$ where $f_{\ell},f_r\in C^\infty(\mathbb R,[0,1])$ satisfy
\begin{equation}\label{eq:fell}
f_{\ell}(x) = \begin{cases} 1 &\text{  for } x\le -1 \\ 
                            0 &\text{  for } x\ge 0, \end{cases}
\qquad
f_{r}(x) = \begin{cases} 0 &\text{  for } x\le 0 \\ 
                         1 &\text{  for } x\ge 1. \end{cases}
\end{equation}                         

The normal  subgroup $\Diff_{c,1}(\mathbb R)= \ker(\on{Shift}_{\ell})=\{\ph=\on{Id}+f: f\in C^\infty_{c,1}(\R),f'>-1\}$ of diffeomorphisms 
which have no shift at $-\infty$ will play an important role subsequently.

\subsection{Groups Related to  $\Diff_{\mathcal S}(\mathbb R)$}\label{diffs} 
The regular Lie group
$\Diff_{\mathcal S}(\mathbb R)$ was treated in \cite[Sect.~6.4]{Michor2006a}. Let us summarize the most important results:
the space 
$\mathcal S(\R)$ consisting of all rapidly decreasing functions is a reflexive nuclear Fr\'echet space.
\newline
{\it The space $C^\infty(\mathbb R,\mathcal S(\mathbb R))$ of smooth
curves in $\mathcal S(\mathbb R)$ consists of all functions 
$f\in C^\infty(\mathbb R^2,\mathbb R)$ satisfying the following 
property:
\begin{enumerate}
\item[$\bullet$]
For all $k,m\in \mathbb N_{\ge0}$ and $n\in \mathbb N_{\ge0}$,  the expression
$(1+|x|^2)^m\p_t^{k}\p^n_xf(t,x)$ is uniformly bounded in $x\in \mathbb R$, locally uniformly bounded  
in $t\in \mathbb R$.
\end{enumerate} 
}
{\it $\Diff_{\mathcal S}(\mathbb R)=\bigl\{\ph=\on{Id}+f: f\in \mathcal S(\mathbb R), 
		f'>-1\bigr\}$ is the corresponding regular Lie group.} 

We again define an extended space:

\noindent
$\mathcal S_2(R)=\{f\in C^\infty(\mathbb R): f'\in \mathcal S(\mathbb R)\}$, 
     the space of antiderivatives of functions in $\mathcal S(\mathbb R)$.
		 It is isomorphic to $\mathbb R\x \mathcal S(\mathbb R)$
     via   $f\mapsto (f(0), f')$. It is again a reflexive nuclear Fr\'echet space,
     contained in $\mathcal B(\mathbb R)$. \newline
{\it The space $C^\infty(\mathbb R,\mathcal S_2(\mathbb R))$ of smooth
curves in $\mathcal S_2(\mathbb R)$ consists of all functions 
$f\in C^\infty(\mathbb R^2,\mathbb R)$ satisfying the following 
property:
\begin{enumerate}
\item[$\bullet$]
For all $k,m\in \mathbb N_{\ge0}$, and $n\in \mathbb N_{>0}$,  the expression
$(1+|x|^2)^m\p_t^{k}\p^n_xf(t,x)$ is uniformly bounded in $x$ and locally uniformly bounded  
in $t$.
\end{enumerate} 
}
We also define the space $\mathcal S_1(\R) = \left\{ f\in\mathcal S_2(\R) \;:\; f(-\infty)=0\right\}$ of antiderivatives of the form $x \mapsto \int_{-\infty}^x g\;dy$ with $g \in \mathcal S(\R)$.

{\it $\Diff_{\mathcal S_2}(\mathbb R)=\bigl\{\ph=\on{Id}+f: f\in \mathcal S_2(\mathbb R), 
		f'>-1\bigr\}$ is the corresponding regular Lie group}.
We have again the short exact sequence of smooth homomorphisms of Lie groups 
$$\xymatrix{
\Diff_{\mathcal S}(\mathbb R)\quad \ar@{>->}[r] & 
\Diff_{\mathcal S_2}(\mathbb R) 
\ar@{->>}[rrr]^{(\on{Shift}_{\ell},\on{Shift}_r)} &&& (\mathbb R^2,+)
},
$$
which splits via the same smooth homomorphic section $s: 
\mathbb R^2\to \Diff_{\mathcal S_2}(\mathbb R)$ as in \ref{diffc} and, thus, describes a 
semi-direct product.
The normal Lie subgroup $\Diff_{\mathcal S_1}(\mathbb R)= \ker(\on{Shift}_{\ell})$ of diffeomorphisms 
which have no shift at $-\infty$ will also play an important role later on.

\subsection{Groups Related to  $\Diff_{W^{\infty,1}}(\mathbb R)$}\label{diffw} The space 
$W^{\infty,1}(\mathbb R)=\bigcap_{k\ge 0}W^{k,1}(\mathbb R) 
= \{f\in C^\infty(\mathbb R): f^{(k)}\in L^1(\mathbb R) \text{  for }k=0,1,2,\dots\}$ 
is the intersection of all Besov spaces of type $L^1$.
It is a reflexive Fr\'echet space. 
By the Sobolev inequality  we have $W^{\infty,1}(\mathbb R)\subset \mathcal B(\mathbb R)$; thus also 
$W^{\infty,1}(\mathbb R)\subset W^{\infty,p}(\mathbb R) = \{f\in C^\infty(\mathbb R): 
f^{(k)}\in L^p(\mathbb R)\text{  for }p=0,1,2,\dots\}$. 
For any $f\in W^{\infty,1}(\mathbb R)$ each 
derivative $f^{(k)}$ is smooth and converges to 0 for $x\to \pm \infty$ since it is in 
$\mathcal B(\mathbb R)$ and
since $C^\infty_c(\mathbb R)$ is dense in $W^{\infty,1}(\mathbb R)$.

{\it The space $C^\infty(\mathbb R,W^{\infty,1}(\mathbb R))$ of smooth
curves $t\mapsto f(t,\cdot)$ in $W^{\infty,1}(\mathbb R)$ consists of all functions 
$f\in C^\infty(\mathbb R^{2},\mathbb R)$ satisfying the following 
property:
\begin{enumerate}
\item[$\bullet$]
For all $k\in \mathbb N_{\ge0}$ and $n\in \mathbb N_{\ge0}$ the expression
$\|\p_t^{k}\p^n_xf(t,\quad)\|_{L^1(\mathbb R)}$ is locally bounded  
in $t$. 
\end{enumerate} 
}
\noindent
{\it $\Diff_{W^{\infty,1}}(\mathbb R)=\bigl\{\ph=\on{Id}+f: f\in W^{\infty,1}(\mathbb R), 
		f'>-1\bigr\}$ 
		denotes the corresponding regular Lie group}.

We again consider an extended space:

\noindent
$W^{\infty,1}_2(\mathbb R)=\{f\in C^\infty(\mathbb R): f'\in W^{\infty,1}(\mathbb R)\}$ 
is the space of bounded antiderivatives of functions in $W^{\infty,1}(\mathbb R)$.
It is isomorphic to $\mathbb R\x W^{\infty,1}(\mathbb R)$ via $f\mapsto (f(0), f')$.
{\it The space $C^\infty(\mathbb R,W^{\infty,1}_2(\mathbb R))$ of smooth
curves in $W^{\infty,1}_2(\mathbb R)$ consists of all functions 
$f\in C^\infty(\mathbb R^2,\mathbb R)$ satisfying the following 
property:
\begin{enumerate}
\item[$\bullet$]
For all $k\in \mathbb N_{\ge0}$, $n\in \mathbb N_{>0}$ and  $t\in \mathbb R$ the expression
$\|\p_t^{k}\p^n_xf(t,\quad)\|_{L^1(\mathbb R)}$ is 
locally bounded in $t$.
\end{enumerate} 
}
We also define the space $W^{\infty,1}_1(\mathbb R)=\{ f\in W^{\infty,1}_2(\mathbb R): 
f(-\infty)=0\}$ of antiderivatives of the form $x\mapsto \int_{-\infty}^x g\,dy$ for 
$g\in W^{\infty,1}(\mathbb R)$.

{\it $\Diff_{W^{\infty,1}_2}(\mathbb R)=\bigl\{\ph=\on{Id}+f: f\in W^{\infty,1}(\mathbb R), 
		f'>-1\bigr\}$ 
		denotes the corresponding regular Lie group}. 

We have again the following exact sequence of smooth homomorphisms of regular Lie groups:
\begin{equation*}
\xymatrix{
\Diff_{W^{\infty,1}}(\mathbb R)\quad \ar@{>->}[r] & 
\Diff_{W^{\infty,1}_2}(\mathbb R) 
\ar@{->>}[rrr]^{(\on{Shift}_{\ell},\on{Shift}_r)} &&& (\mathbb R^2,+)
},
\end{equation*}
which splits with the same section as for 
$\Diff_{c,2}(\mathbb R)$.
The group $\Diff_{W^{\infty,1}_1}(\mathbb R)= \ker(\on{Shift}_{\ell})$ of diffeomorphisms, which 
have no shift at $-\infty$,  will play an important role later on. 

\subsection*{Remark on the $H^\infty=W^{\infty,2}$-case} 
One may wonder why we use the groups related to $\Diff_{W^{\infty,1}}(\mathbb R)$ instead of those 
modelled on the more usual intersection $H^\infty$ of all Sobolev spaces. The reason is that 
$H^\infty\not\subset L^1$; thus, for $f\in H^\infty(\mathbb R)$ 
the antiderivative $x\mapsto \int_{-\infty}^x f(y)\,dy$ is not bounded in general, and  
the extended groups are not contained in 
$\Diff_{\mathcal B}(\mathbb R)$ and thus do not admit charts. 
If we model the groups on 
$H^\infty\cap L^1$, then they are {\bf not} Lie groups: right translations are smooth, but left
translations are not. See 
\cite[Sect.~14]{KMR14} for this surprising fact.

\begin{thm*}\label{regularLie} 
All the groups $\Diff_c(\mathbb R)$, 
$\Diff_{c,1}(\mathbb R)$, 
$\Diff_{c,2}(\mathbb R)$, 
$\Diff_{\mathcal S}(\mathbb R)$, 
$\Diff_{\mathcal S_1}(\mathbb R)$, 
$\Diff_{\mathcal S_2}(\mathbb R)$, 
$\Diff_{W^{\infty,1}}(\mathbb R)$, 
$\Diff_{W^{\infty,1}_1}(\mathbb R)$, 
$\Diff_{W^{\infty,1}_2}(\mathbb R)$, 
and 
$\Diff_{\mathcal B}(\mathbb R)$
are  smooth regular Lie groups. 
We have the following smooth injective group homomorphisms:
$$\xymatrix{
\Diff_c(\mathbb R) \ar[r] \ar[d] & \Diff_{\mathcal S}(\mathbb R) \ar[d] \ar[r] & 
\Diff_{W^{\infty,1}}(\mathbb R) \ar[d] 
\\
\Diff_{c,1}(\mathbb R) \ar[r] \ar[d] & \Diff_{\mathcal S_1}(\mathbb R) \ar[r] \ar[d] &  
\Diff_{W^{\infty,1}_1}(\mathbb R) \ar[d]
\\
\Diff_{c,2}(\mathbb R) \ar[r] & \Diff_{\mathcal S_2}(\mathbb R) \ar[r] &  
\Diff_{W^{\infty,1}_2}(\mathbb R) \ar[r] & \Diff_{\mathcal B}(\mathbb R)
}$$
Each group is a normal subgroup in any other in which it is contained, in particular in 
$\on{Diff}_{\mathcal B}(\mathbb R)$.
\end{thm*}

\begin{proof}
That the groups  $\Diff_{c}(\mathbb R)$, $\Diff_{\mathcal S}(\mathbb R)$, 
and $\Diff_{\mathcal B}(\mathbb R)$ are 
regular Lie groups is proved (for $\mathbb R^n$ instead of $\mathbb R$) in 
\cite{Michor2012b_preprint}, and in \cite[Sect.~43.1]{Michor1997} for 
$\Diff_{c}(\mathbb R)$. Moreover, the group $\Diff_{H^\infty}(\mathbb R^n)$ (where $H^\infty= 
W^{\infty,2}$) is treated in \cite{Michor2012b_preprint}; the proof for $W^{\infty,1}$ is the same. 
See also \cite[Theorem~8.1]{KMR14}.
The extension to the semi-direct products is easy and is proved in 
\cite[Sect.~38.9]{Michor1997}. That each group is normal in the largest one is also proved in 
\cite{Michor2012b_preprint}.
\end{proof}

\subsection{A Remark on the Existence of Normal Subgroups}
This section will not be used in the remainder of the paper. 
It is a well-known result that $\Diff_c(\mathbb R)$ is a simple group (see \cite{Mather1974}, \cite{Mather1975}, 
\cite{Mather1985}). We want to discuss some effects of this result for the diffeomorphism groups introduced in the previous sections.

\emph{Existence of normal subgroups in $\Diff_{c,2}(\mathbb R)$:} We first claim that 
any non-trivial normal subgroup $N$ of $\Diff_{c,2}(\mathbb R)$ intersects $\Diff_c(\mathbb R)$ 
non-trivially: let $\on{Id}\ne \ph\in N$. If $\ph$ has compact support, we are done.
If $\ph$ does not have compact support, without loss of generality assume that 
$\on{Shift}_r(\ph) = a >0$. Thus, for some $x_0$ we have $\ph(x)=x+a$ for $x\ge x_0$. Choose 
$x_0+2a<x_1<x_2<x_0+3a$ and $\ps\in\Diff_c(\mathbb R)$ with support in the interval $[x_1,x_2]$, so that 
$\ps(x)=x$ for $x\notin [x_1,x_2]$.
For $x\in [x_1,x_2]$ we then have  
$(\ps\i\o\ph\o\ps)(x) = \ps\i(\ps(x)+a) = \ps(x)+a.$ 
Thus, $\ps\i\o\ph\o\ps\in N$ differs from $\ph$ just on the compact interval $[x_1,x_2]$. But 
then $\ps\i\o\ph\o\ps\o\ph\i\in N$ has compact support, and we are done.

By simplicity of $\Diff_c(\mathbb R)$ we obtain $N\supseteq \Diff_c(\mathbb R)$. Therefore, the lattice of 
normal subgroups of $\Diff_{c,2}(\mathbb R)$ has a $\Diff_c(\mathbb R)$ as minimal element and, thus, is
isomorphic to the lattice of subgroups of $(\mathbb R^2,+)$, which is quite large (see 
\cite{Fuchs1970}, \cite{Fuchs1973}).

\emph{Existence of normal subgroups $N$  with
$\Diff_c(\mathbb R)\to N \to \Diff_{\mathcal S}(\mathbb R)$:}
By conjugating with $x\mapsto 1/x$ we see that the quotient group 
$\Diff_{\mathcal S}(\mathbb R)/\Diff_c(\mathbb R)$ is isomorphic to the group of germs at 0 of 
smooth diffeomorphisms $\ph:(\mathbb R,0)\to (\mathbb R,0)$ such that $\ph(x)-x$ is flat at 0:
$\ph(x)-x = o(|x|^N)$ for each $N$. This group contains infinitely many normal subgroups:
$\ph(x)-x=o(e^{-1/|x|})$ or $=o(\exp(-\exp(1/|x|)))$, and so on.

We have not looked for normal subgroups $N$ with  
$$
\Diff_{W^{\infty,1}}(\mathbb R) \to N 
\to \Diff_{W^{\infty,1}_2}(\mathbb R).
$$

\subsection{Groups of Real-Analytic Diffeomorphisms}\label{diffom}
Since the HS equation will turn out to have solutions in smaller groups of 
diffeomorphisms, we give here a description of them. For simplicity's sake, we restrict our 
attention to the periodic case. Let $\Diff^\om(S^1)$ be the real-analytic regular Lie group of all 
real-analytic diffeomorphisms of $S^1$, with the real-analytic structure described in 
\cite[Sect.~8.11]{Kriegl1990}, see also 
\cite[theorem 43.4]{Michor1997}. Recall that a mapping between $c^\infty$-open sets 
of convenient vector spaces is real-analytic if preserves smooth curves and also real-analytic curves.

\subsection{Groups of Ultra-Differentiable Diffeomorphisms}\label{diffDC}
Let us now describe the Denjoy--Carleman ultra-differentiable function classes which admit convenient 
calculus, following \cite{Kriegl2009}, \cite{Kriegl2011a}, and \cite{Michor2012c_preprint}. 
We consider a sequence $M=(M_k)$ of positive real numbers serving as weights for 
derivatives. For a smooth function $f$ on an open subset $U$ in $\mathbb R^n$, a compact set 
$K\subset U$, and for $\rh>0$  consider the set
\begin{equation}\label{eq:diffDC}
\Big\{\frac{d^kf(x)}{\rh^{k} \, k! \, M_{k}} : x \in K, k \in \mathbb N \Big\}.
\end{equation}
We define the \emph{Denjoy--Carleman classes} 
\begin{align*}
  C^{(M)}(U) &:= \{f \in C^\infty(U) : \forall \text{ compact } K \subseteq U ~\forall \rh>0: 
  \eqref{eq:diffDC} \text{ is bounded} \}, \\
  C^{\{M\}}(U) &:= \{f \in C^\infty(U) : \forall \text{ compact } K \subseteq U ~\exists \rh>0: 
  \eqref{eq:diffDC} \text{ is bounded} \}.
\end{align*}
The elements of $C^{(M)}(U)$ are said to be of \emph{Beurling type}; those of $C^{\{M\}}(U)$ of \emph{Roumieu type}. 
If $M_k=1$, for all $k$, then $C^{(M)}(\mathbb R)$ consists of the restrictions to $U$ 
of the real and imaginary parts of all entire functions, 
while $C^{\{M\}}(\mathbb R)$ coincides with the ring $C^\om(\mathbb R)$ of real-analytic functions.
We shall also write $C^{[M]}$ to mean either $C^{(M)}$ or $C^{\{M\}}$. 
We shall assume that the sequence $M=(M_k)$ has the following properties:
\begin{itemize}
	\item $M$ is log-convex: $k\mapsto\log(M_k)$ is convex, i.e. 
        $M_k^2 \le M_{k-1} \, M_{k+1}$ for all $k$.
  \item $M$ has moderate growth, i.e. 
	      $\sup_{j,k \in \mathbb N_{>0}} \Big(\frac{M_{j+k}}{M_j \, M_k}\Big)^{\frac{1}{j+k}} < \infty.$ 
	\item	In the Beurling case $C^{[M]} = C^{(M)}$ we also require that $C^\om\subseteq C^{(M)}$, 
        or equivalently $M_k^{1/k} \to \infty.$
\end{itemize}
Then, both classes $C^{[M]}$  are closed under composition and differentiation,  can be extended
to convenient vector spaces, and  form monoidally closed categories (i.e. admit convenient settings).
Moreover,  on open sets in $\mathbb R^n$, $C^{[M]}$-vector fields have $C^{[M]}$-flows, 
and 	between Banach spaces, the $C^{[M]}$ implicit function theorem holds. 

For mappings between $c^\infty$-open subsets of convenient vector spaces we have the following statements: 
\begin{itemize}
	\item For non-quasi-analytic $M$, the mapping $f$ is $C^{\{M\}}$ if it maps $C^{\{M\}}$-curves to 
        $C^{\{M\}}$-curves, by \cite{Kriegl2009}.
	\item For certain quasi-analytic $M$, the mapping $f$ is $C^{\{M\}}$ if it maps $C^{\{N\}}$-curves 
        to $C^{\{N\}}$-curves, for all non-quasi-analytic $N$ which are larger than $M$ and have the 
        aforementioned properties, by \cite{Kriegl2011a}.
	\item For any $M$, the mapping $f$ is $C^{[M]}$ if it respects $C^{[M]}$-maps from open balls in 
        Banach spaces, by \cite{Michor2012c_preprint}. 
\end{itemize}
For every $M$ with the aforementioned properties we have the regular Lie groups $\Diff^{\{M\}}(S^1)$ and 
$\Diff^{[M]}(S^1)$ (we write $\Diff^{[M]}(S^1)$ if we mean any of the two) of 
$C^{[M]}$-diffeo\-mor\-phisms of $S^1$ which is a $C^{[M]}$-group (but not better), by 
\cite[Sect.~6.5]{Kriegl2009}, \cite[Sect.~5.6]{Kriegl2011a}, and \cite[Sect.~9.8]{Michor2012c_preprint}. 

\section{Right-invariant Riemannian Metrics on Lie Groups}

\subsection{Notation on Lie Groups}
Let $G$ be a regular Lie group, which may be infinite-dimensional, with Lie
algebra $\mathfrak g$. 
Let $\mu:G\x G\to G$ be the group multiplication, $\mu_x$ the left
translation and $\mu^y$ the right translation, 
given by $\mu_x(y)=\mu^y(x)=xy=\mu(x,y)$. 

Let $L,R:\mathfrak g\to \X(G)$ be the left- 
and right-invariant vector field mappings, given by 
$L_X(g)=T_e(\mu_g).X$ and $R_X=T_e(\mu^g).X$ respectively. 
They are related by $L_X(g)=R_{\on{Ad}(g)X}(g)$.
Their flows are given by 
\begin{displaymath}
\on{Fl}^{L_X}_t(g)= g.\exp(tX)=\mu^{\exp(tX)}(g),\quad
\on{Fl}^{R_X}_t(g)= \exp(tX).g=\mu_{\exp(tX)}(g).
\end{displaymath}

We also need the right
Maurer--Cartan form $\ka=\ka^r\in\Om^1(G,\mathfrak g)$, given by  
$\ka_x(\xi):=T_x(\mu^{x\i})\cdot \xi$. It satisfies the left
Maurer--Cartan equation $d\ka-\tfrac12[\ka,\ka]_\wedge=0$, where
$[\quad,\quad]_\wedge$ denotes the wedge product of $\mathfrak g$-valued forms on
$G$ induced by the Lie bracket. Note that
$\tfrac12[\ka,\ka]_\wedge (\xi,\et) = [\ka(\xi),\ka(\et)]$.
The (exterior) derivative of the function $\on{Ad}:G\to GL(\mathfrak g)$ can be 
expressed by
\begin{displaymath}
d\on{Ad} = \on{Ad}.(\on{ad}\o\ka^l) = (\on{ad}\o \ka^r).\on{Ad}
\end{displaymath}
since we have
$d\on{Ad}(T\mu_g.X) = \frac d{dt}|_0 \on{Ad}(g.\exp(tX))
= \on{Ad}(g).\on{ad}(\ka^l(T\mu_g.X))$.

\subsection{Geodesics of a Right-Invariant Metric on a Lie  
Group}\label{adjoint} 
 
Let $\ga=\mathfrak g\x\mathfrak g\to\mathbb R$ be a positive-definite bounded (weak) inner product. Then  
\begin{equation}
\ga_x(\xi,\et)=\ga\big( T(\mu^{x\i})\cdot\xi,\, T(\mu^{x\i})\cdot\et\big) =  
     \ga\big(\ka(\xi),\,\ka(\et)\big)
\end{equation}
is a right-invariant (weak) Riemannian metric on $G$, and any
(weak) right-invariant bounded Riemannian metric is of this form, for
suitable  $\ga$.
We shall denote by $\check\ga:\mathfrak g\to \mathfrak g^*$ the mapping induced by $\ga$ from the 
Lie algebra into its dual (of bounded linear functionals) and by 
$\langle \al, X \rangle_{\mathfrak g}$ the duality evaluation between $\al\in\mathfrak g^*$ and 
$X\in \mathfrak g$.
 
Let $g:[a,b]\to G$ be a smooth curve.  
The velocity field of $g$, viewed in the right trivializations, 
coincides  with the right logarithmic derivative
\begin{displaymath}
\de^r(g)=T(\mu^{g\i})\cdot \partial_t g =  
\ka(\partial_t g) = (g^*\ka)(\partial_t),
\text{ where }
\partial_t=\frac{\partial}{\partial t}. 
\end{displaymath}
The energy of the curve $g(t)$ is given by  
\begin{equation}
\label{eq:enG}
E(g) = \frac12\int_a^b\ga_g(g',g')dt = \frac12\int_a^b 
     \ga\big( (g^*\ka)(\partial_t),(g^*\ka)(\partial_t)\big)\, dt. 
\end{equation}
For a variation $g(s,t)$ with fixed endpoints we then use that
$$d(g^*\ka)(\p_t,\p_s)=\p_t(g^*\ka(\p_s))-\p_s(g^*\ka(\p_t))-0,$$ partial integration and the
left Maurer--Cartan equation to obtain
\begin{align*} 
&\partial_sE(g) = \frac12\int_a^b2 
     \ga\big( \partial_s(g^*\ka)(\partial_t),\,
                            (g^*\ka)(\partial_t)\big)\, dt
\\&
= \int_a^b \ga\big( \partial_t(g^*\ka)(\partial_s) - 
         d(g^*\ka)(\partial_t,\partial_s),\,
                (g^*\ka)(\partial_t)\big)\,dt
\\&
= -\int_a^b \ga\big( (g^*\ka)(\partial_s),\,\partial_t(g^*\ka)(\partial_t)\big)\,dt 
  - \int_a^b \ga\big( [(g^*\ka)(\partial_t),(g^*\ka)(\partial_s)],\, (g^*\ka)(\partial_t)\big)\, dt
\\&
= -\int_a^b \big\langle\check\ga(\partial_t(g^*\ka)(\partial_t)),\,
  (g^*\ka)(\partial_s)\big\rangle_{\mathfrak g} \, dt
\\&\qquad\qquad
  - \int_a^b \big\langle  \check\ga((g^*\ka)(\partial_t)),\, 
  \on{ad}_{(g^*\ka)(\partial_t)}(g^*\ka)(\partial_s)\big\rangle_{\mathfrak g} \, dt
\\&
= -\int_a^b 
\big\langle\check\ga(\partial_t(g^*\ka)(\partial_t)) + (\on{ad}_{(g^*\ka)(\partial_t)})^{*}\check\ga((g^*\ka)(\partial_t)),\, 
  (g^*\ka)(\partial_s)\big\rangle_{\mathfrak g} \, dt.
\end{align*}
Thus the curve $g(0,t)$ is critical for the energy \eqref{eq:enG} if and only if
$$
\check\ga(\partial_t(g^*\ka)(\partial_t)) + 
(\on{ad}_{(g^*\ka)(\partial_t)})^{*}\check\ga((g^*\ka)(\partial_t)) = 0.
$$
In terms of the right logarithmic derivative $u:[a,b]\to \mathfrak g$ of  
$g:[a,b]\to G$, given by  
$u(t):= g^*\ka(\partial_t) = T_{g(t)}(\mu^{g(t)\i})\cdot g'(t)$, 
the geodesic equation has the expression
\begin{equation}
\boxed{\quad
\p_t u = - \,\check\ga\i\on{ad}(u)^{*}\;\check\ga(u)\quad} 
\end{equation}
Thus the geodesic equation exists in general if and only if 
$\on{ad}(X)^{*}\check\ga(X)$ is in the image of $\check\ga:\mathfrak g\to\mathfrak g^*$, i.e.
\begin{equation}
\label{eq:adim}
\on{ad}(X)^{*}\check\ga(X) \in \check\ga(\mathfrak g)
\end{equation}
 for every 
$X\in\mathfrak X$. Condition \eqref{eq:adim} then leads to the existence of the so-called 
Christoffel symbols. Interestingly, it is not neccessary for the more restrictive condition 
$\on{ad}(X)^{*}\check\ga(Y) \in \check \ga \in \mathfrak g$ to be satisfied in order to obtain the 
geodesic equation, Christoffel symbols and the curvature; compare with \cite[Lemma 
3.3]{Michor2006a}. Note here the appearance of the geodesic equation for the {\it momentum} $p:=\ga(u)$:
$$
p_t = - \on{ad}(\check\ga\i(p))^*p.
$$
Subsequently, we shall encounter situations where \eqref{eq:adim} is satisfied but where the usual 
transpose $\on{ad}^\top(X)$ of $\on{ad}(X)$, 
\begin{equation}
\label{eq:ad_trans}
\on{ad}^\top(X) := \check\ga\i\o\on{ad}_X^*\o\; \check\ga
\end{equation}
does not exist for all $X$. 

\subsection{Covariant Derivative} 
The right trivialization $(\pi_G,\ka^r):TG\to G\x{\mathfrak g}$  
induces the isomorphism $R:C^\infty(G,{\mathfrak g})\to \X(G)$, given by  
$R(X)(x):= R_X(x):=T_e(\mu^x)\cdot X(x)$, for $X\in C^\infty(G,{\mathfrak g})$ and
$x\in G$. Here $\X(G):=\Ga(TG)$ denotes the Lie algebra of all vector 
fields. For the Lie bracket and the Riemannian metric we have 
\begin{align*} 
[R_X,R_Y] &= R(-[X,Y]_{\mathfrak g} + dY\cdot R_X - dX\cdot R_Y),
\\
R\i[R_X,R_Y] &= -[X,Y]_{\mathfrak g} + R_X(Y) - R_Y(X),
\\
\ga_x(R_X(x),R_Y(x)) &= \ga( X(x),Y(x))\,,\, x\in G. 
\end{align*}
In what follows, we shall perform all computations in $C^\infty(G,{\mathfrak g})$ instead of 
$\X(G)$. In particular, we shall use the convention
\begin{displaymath}
\nabla_XY := R\i(\nabla_{R_X}R_Y)\quad\text{ for }X,Y\in C^\infty(G,{\mathfrak g}).
\end{displaymath}
to express the Levi-Civita covariant derivative. 
 
\begin{lem*}
Assume that for all $\xi\in{\mathfrak g}$ the element $\on{ad}(\xi)^*\check\ga(\xi)\in\mathfrak g^*$ is in the image of
$\check\ga:\mathfrak g\to\mathfrak g^*$   
and that  
$\xi\mapsto \check\ga\i\on{ad}(\xi)^*\check\ga(\xi)$ is bounded  quadratic (or, equivalently, smooth). 
Then the Levi-Civita covariant derivative of the metric $\ga$ 
exists and is given for any $X,Y \in C^\infty(G,{\mathfrak g})$  in
terms of the isomorphism $R$ by
\begin{equation*}
\nabla_XY= dY.R_X + \rh(X)Y - \frac12\on{ad}(X)Y,
\end{equation*}
where 
\[
\rh(\xi)\et = \tfrac14\check\ga\i\big(\on{ad}_{\xi+\et}^*\check\ga(\xi+\et) - \on{ad}_{\xi-\et}^*\check\ga(\xi-\et)\big) = \tfrac12\check\ga\i\big(\on{ad}_\xi^*\check\ga(\et) + \on{ad}_\et^*\check\ga(\xi)\big)
\]
is the polarized version. 
$\rh:{\mathfrak g}\to L({\mathfrak g},{\mathfrak g})$ is bounded, and we have $\rh(\xi)\et=\rh(\et)\xi$.  
We also have
\begin{gather*}
\ga\big(\rh(\xi)\et,\ze\big) = \frac12\ga(\xi,\on{ad}(\et)\ze) + \frac12\ga(\et,\on{ad}(\xi)\ze),
\\
\ga(\rh(\xi)\et,\ze) + \ga(\rh(\et)\ze,\xi) + \ga(\rh(\ze)\xi,\xi) = 0.
\end{gather*}
\end{lem*} 

\begin{proof}
It is easily checked that $\nabla$ is a covariant derivative. The Riemannian
metric is covariantly constant since
\begin{align*}
R_X\ga( Y,Z) &= \ga( dY.R_X,Z) + \ga( Y,dZ.R_X)
=\ga( \nabla_XY,Z) + \ga( Y,\nabla_XZ).  
\end{align*}
Since $\rh$ is symmetric, the connection is also torsion-free:
\begin{displaymath}
\nabla_XY-\nabla_YX + [X,Y]_{\mathfrak g} - dY.R_X + dX.R_Y = 0.
\end{displaymath}
\end{proof}

\subsection{Curvature} 
For $X,Y\in C^\infty(G,{\mathfrak g})$ we have
$$ 
[R_X,\on{ad}(Y)] = \on{ad}(R_X(Y))\quad\text{  and }\quad  
[R_X,\rh(Y)] = \rh(R_X(Y)).
$$
The Riemannian curvature is then computed by  
\begin{align*} 
\mathcal{R}(X,Y) &=  
[\nabla_X,\nabla_Y]-\nabla_{-[X,Y]_{\mathfrak g}+R_X(Y)-R_Y(X)}
\\&
= [R_X+\rh_X-\tfrac12\on{ad}_X, 
     R_Y+\rh_Y-\tfrac12\on{ad}_Y]
\\&\quad
- R({-[X,Y]_{\mathfrak g} + R_X(Y) - R_Y(X)}) 
-\rh({-[X,Y]_{\mathfrak g} + R_X(Y)\! - R_Y(X)})
\\&\quad
+\frac12\on{ad}(-[X,Y]_{\mathfrak g} + R_X(Y)\! - R_Y(X)) 
\\&
= [\rh_X,\rh_Y] +\rh_{[X,Y]_{\mathfrak g}} -\frac12[\rh_X,\on{ad}_Y] +\frac12[\rh_Y,\on{ad}_X] 
-\frac14\on{ad}_{[X,Y]_{\mathfrak g}}.
\end{align*}
which is visibly a tensor field.

\subsection{Sectional Curvature}\label{sectionalcurvature}
For the linear two-dimensional subspace $P\subseteq {\mathfrak g}$, that is spanned by linearly independent $X,Y\in {\mathfrak g}$, the 
sectional curvature is defined as
$$
k(P) = -\frac{\ga\big(\mathcal R(X,Y)X,Y\big)}{\|X\|_\ga^2\|Y\|_\ga^2-\ga(X,Y)^2}.
$$
For the numerator we obtain
\begin{align*}
\ga\big(\mathcal R(&X,Y)X,Y\big) = \ga(\rh_X\rh_YX,Y) - \ga(\rh_Y\rh_XX,Y) + \ga(\rh_{[X,Y]}X,Y)
\\&\quad
-\frac12\ga(\rh_X[Y,X],Y) + \frac12\ga([Y,\rh_XX],Y) 
\\&\quad
+0 - \frac12\ga([X,\rh_YX],Y) -\frac14\ga([[X,Y],X],Y)
\\&
= \ga(\rh_XX,\rh_YY) - \|\rh_XY\|_\ga^2 + \frac34\|[X,Y]\|_\ga^2  
\\&\quad
-\frac12\ga(X,[Y,[X,Y]]) + \frac12\ga(Y,[X,[X,Y]])
\\&
= \ga(\rh_XX,\rh_YY) - \|\rh_XY\|_\ga^2 + \frac34\|[X,Y]\|_\ga^2  
\\&\quad
-\ga(\rh_XY,[X,Y]]) + \ga(Y,[X,[X,Y]]).
\end{align*}
If the adjoint $\on{ad}(X)^\top:\mathfrak g\to \mathfrak g$ exists, this is easily seen to coincide with 
Arnold's original formula \cite{Arnold1966},
\begin{align*}
&\ga(\mathcal R(X,Y)X,Y) = 
- \frac14\|\on{ad}(X)^\top Y+\on{ad}(Y)^\top X\|^2_\ga
+ \ga(\on{ad}(X)^\top X,\on{ad}(Y)^\top Y)   
\\&
+ \frac12\ga(\on{ad}(X)^\top Y-\on{ad}(Y)^\top X,\on{ad}(X)Y)  
+ \frac34\|[X,Y]\|_\ga^2.
\end{align*}

\section{Homogeneous $H^1$-Metric on $\Diff(\R)$ and the Hunter--Saxton Equation}\label{section4}

In this section we will study the homogeneous $H^1$ or $\dot H^1$-metric on the various diffeomorphism groups of $\R$ defined in Sect. \ref{diffgroups}. It was shown in  \cite{Misiolek2003} that the geodesic equation of the $\dot H^1$-metric on $\on{Diff}(S^1)$ is the  HS equation. We will show that suitable diffeomorphism groups on the real line also have the HS equation as geodesic equation. In \cite{Lenells2007} a way was found to map $\on{Diff}(S^1)$ isometrically to an open subset of an $L^2$-sphere in $C^\infty(S^1,\R)$. We will generalize this representation to the non-periodic case.

In the situation studied here --- diffeomorphism groups on the real line --- the resulting geometry will be different from the periodic case. Some of the diffeomorphism groups will be flat in the sense of Riemannian geometry, while others will be  submanifolds of a flat space (see \cite[Sect. 27.11]{Kriegl1997} for the definiton of (splitting) submanifolds in an infinite dimensional setting).

\subsection{Setting}\label{setting}
In this section, we shall use any of the following regular Lie groups:
\begin{enumerate}
 \item We will denote by $\A(\R)$ any of the spaces $C^\infty_c(\R)$, $\mathcal S(\R)$, 
 $W^{\infty,1}(\R)$, or $H^{\infty}_0(\mathbb R)$. By $\on{Diff}_\A(\R)$ we will denote
 the corresponding groups $\Diff_c(\R)$, $\Diff_{\mathcal S}(\R)$, and $\Diff_{W^{\infty,1}}(\R)$.
  as defined in  Sects.  \ref{diffc}, \ref{diffs} and \ref{diffw} respectively.
 \item Similarly $\AA(\R)$ will denote any of the spaces $C^\infty_{c,1}(\R)$, $\mathcal S_1(\R)$, 
 or $W^{\infty,1}_1(\R)$. By $\on{Diff}_\AA(\R)$ we will denote the corresponding groups
 $\Diff_{c,1}(\R)$, $\Diff_{\mathcal S_1}(\R)$, or $\Diff_{W^{\infty,1}_1}(\R)$. 
  as defined in  Sects.  \ref{diffc}, \ref{diffs} and \ref{diffw} respectively.
\end{enumerate}

\subsection{$\dot H^1$-Metric}\label{H1metric}
For $\Diff_{\A}(\mathbb R)$ and $\Diff_{\AA}(\mathbb R)$ the homogeneous $H^1$-metric is given by
$$G_{\ph}(X\circ\ph,Y\circ\ph) = G_{\on{Id}}(X,Y)= \int_{\R} X'(x)Y'(x)\;dx,$$
where $X, Y$ are elements of the Lie algebra $\A(\R)$ or $\AA(\R)$. We shall also use the notation
$$\langle\cdot,\cdot\rangle_{\dot H^1}:=G(\cdot,\cdot).$$

\begin{thm*} 
On $\on{Diff}_\AA(\R)$ the geodesic equation is the HS equation
\begin{equation}
\label{eq:hs_geod}
\boxed{\begin{aligned}
(\ph_t)\circ\ph\i=u  \qquad u_{t} = -u u_x +\frac12 \int_{-\infty}^x   (u_x(z))^2 \,dz,
\end{aligned}}
\end{equation}
and the induced geodesic distance is positive. 

On the other hand, the geodesic equation does not exist on the subgroup $\Diff_{\A}(\mathbb R)$, since the adjoint 
$\on{ad}(X)^*\check G_{\on{Id}}(X)$ does not lie in $\check G_{\on{Id}}(\A(\R))$ for all $X\in\A(\R)$. 
\end{thm*}
Note that this is a natural example of a non-robust Riemannian manifold in the sense of \cite[Sect.~2.4]{Michor2012a_preprint}.

\begin{proof} Note that $\check G_{\on{Id}}: \A_1(\R)\to \A_1(\R)^\ast$ is given by $\check G_{\on{Id}}(X) = -X''$ if we use the $L^2$-pairing $X\mapsto (Y\mapsto \int XYdx)$ to embed functions into the space of distributions.
We now compute the adjoint of the operator $\on{ad}(X)$ as defined in Sect.~\ref{adjoint}. 
\begin{align*}
\big\langle \on{ad}(X)&^*\check G_{\on{Id}}(Y),Z\big\rangle = \check G_{\on{Id}}(Y,\on{ad}(X)Z) = G_{\on{Id}}(Y,-[X,Z]) 
\\&
=\int_{\mathbb R} Y'(x)\big(X'(x)Z(x)-X(x)Z'(x)\big)'\,dx\
\\&
=\int_{\mathbb R} Z(x)\big(X''(x)Y'(x) - (X(x)Y'(x))''\big) \,dx.
\end{align*}
Therefore, the adjoint as an element of $\A_1^\ast$ is given by  
\begin{align*}
\on{ad}(X)^* \check G_{\on{Id}}(Y) = X''Y' -(XY')'' .
\end{align*}
For $X=Y$ we can rewrite this as
\begin{align*}
\on{ad}(X)^* \check G_{\on{Id}}(X)&= \tfrac12\big((X'^2)' - (X^2)'''\big)
= \frac12\Big(\int_{-\infty}^x X'(y)^2\,dy -(X^2)' \Big)''
\\&
= \frac12\check G_{\on{Id}}\Big(-\int_{-\infty}^x X'(y)^2\,dy +(X^2)' \Big).
\end{align*}
If $X \in \AA(\R)$ then the function $-\tfrac12\int_{-\infty}^x X'(y)^2\,dy +\tfrac12(X^2)'$ is again an element of $\AA(\R)$. This follows immediately from the definition of $\AA(\R)$.  Therefore, the geodesic equation exists on $\on{Diff}_\AA(\R)$ and is given by \eqref{eq:hs_geod}.

However, if $X \in \A(\R)$, a neccessary condition for $\int_{-\infty}^x (X'(y))^2 dy \in \A(\R)$ would be $\int_{-\infty}^\infty X'(y)^2 dy = 0$, which would imply $X'=0$. Thus, the geodesic equation does not exist on $\A(\R)$.

The positivity of geodesic distance will follow from the explicit formula given in Corollary \ref{geod_dist}. 
\end{proof}

\subsection*{Remark}
One obtains the classical form of the HS equation 
\begin{align*}
u_{tx} = -u u_{xx}-\frac12 u^2_x ,
\end{align*} 
by differentiating the preceding geodesic Eq. \eqref{eq:hs_geod}.
In Sect.~\ref{RmapDiffAA} we will use a geometric argument to give an explicit solution formula, 
which will also imply the well-posedness of the equation. For $\AA(\R)=W^{\infty,1}_1(\R)$ 
an analytic proof of well-posedness could also be carried out similarly 
to that in \cite[Sect. 10]{Bauer2011b} by adapting the arguments to $\R$.
Furthermore, using this geometric trick, we will
conclude that the curvature of the $\dot H^1$ metric vanishes. One can also show this statement 
directly using the adaption of Arnold's formula presented in Sect.~\ref{sectionalcurvature}. From 
the foregoing proof one can easily deduce the formula for the mapping $\rh$: 
\begin{align*}
 \rho(X)Y&=\frac12\check G^{-1}\left(\on{ad}(X)^* G(Y)+\on{ad}(X)^* G(Y)\right)\\&=
 \frac12\check G^{-1}\left(X''Y'-(XY')''+ Y''X'-(YX')''\right)\\&=
  \frac12\check G^{-1}\left((X'Y')'+(XY)'''\right)  \\&=
  \frac12 \left(-\int_{-\infty}^x (X'Y')\;dx + (XY)'\right)
\end{align*}
Using this formula, the desired formula for the curvature is a straightforward calculation.

\subsection*{Remark}
For general $X \neq Y \in \AA(\R)$ we will have $\on{ad}(X)^* \check G_{\on{Id}}(Y) 
\notin \check G_{\on{Id}}(\AA(\R))$. If there were a function $Z \in \AA(\R)$ such that 
$-Z'' = X''Y' - (XY')''$ then a necessary condition would be $0 = Z'(-\infty) = 
-\int_{-\infty}^\infty X''Y' dx$, which would in general not be satisfied. Thus, the transpose of 
$\on{ad}(X)$ as defined in \eqref{eq:ad_trans} does not exist; only the symmetric version $X 
\mapsto \check G_{\on{Id}}\i(\on{ad}(X)^\ast\check G_{\on{Id}} (X))$ exists.

\subsection{The Square-Root Representation on $\Diff_{\AA}(\R)$.}\label{RmapDiffAA}

We will define a map $R$ from $\on{Diff}_\AA(\R)$ to the space
\[
\A(\R,\R_{>-2}) = \left\{ f \in \A(\R) \;:\; f(x) > -2 \right\},
\]
such that the pullback of the $L^2$-metric on $\A(\R,\R_{>-2})$ is the $\dot H^1$-metric on the space $\on{Diff}_\AA(\R)$. Since $\A(\R,\R_{>-2})$ is an open subset of $C^\infty(\R)$ this implies that $\on{Diff}_\AA(\R)$ with the $\dot H^1$-metric is a flat space in the sense of Riemannian geometry. This is an adaptation of the square-root representation of $\on{Diff}(S^1)/S^1$ used in \cite{Lenells2007}; see also Sect.~\ref{periodic}, where we review this construction.

\begin{thm*}
We define the $R$-map by
$$ R:\left\{
\begin{aligned}
 \Diff_{\AA}(\mathbb R)&\to 
\A\big(\R,\mathbb R_{>-2}\big)\subset\A(\R,\R)\\
\ph &\mapsto
2\;\big((\ph')^{1/2}-1\big) .
\end{aligned}\right.
$$
The $R$-map is invertible with inverse
$$R\i :\left\{
\begin{aligned}\A\big(\R,\mathbb R_{>-2}\big) &\to \Diff_{\AA}(\mathbb R)
\\\ga&\mapsto x+\frac14 \int_{-\infty}^x \big(\ga^2+4\ga\big)\;dx .
\end{aligned}\right.$$ 
The pullback of the flat $L^2$-metric via $R$ is the $\dot H^1$-metric on $\on{Diff}_\AA(\R)$, i.e.
$$R^*\langle \cdot,\cdot\rangle_{L^2} = \langle\cdot,\cdot\rangle_{\dot H^1} .$$
Thus the space $\big(\Diff_{\AA}(\R),\dot H^1\big)$ is a flat space in the sense of Riemannian geometry.
\end{thm*}
Here $\langle \cdot,\cdot\rangle_{L^2}$ denotes the $L^2$-inner product on $\A(\R)$ interpreted as a Riemannian metric on $\A(\R,\R_{>-2})$, which does not depend on the basepoint, i.e.
\[
G^{L^2}_\ga(h,k) = \langle h,k\rangle_{L^2}=\int_\R h(x)k(x)\;dx,
\]
for $h,k \in \A(\R) \cong T_\ga \A(\R,\R_{>-2})$.

\begin{proof}
We will first prove that for $\ph\in \Diff_{\AA}(\mathbb R)$ the image $R(\ph)$ is an element of $\A(\R,\R_{>-2})$. 
To do so, we write  $\ph(x)=x+f(x)$, with $f\in\AA$. Using a  Taylor expansion of $\sqrt{1+x}$ around $x=0$,
\[
\sqrt{1+x} = 1 + \frac 12 x - \frac 14 \int_0^1 \frac {1-t}{(1+tx)^{\frac 32}} dt\; x^2,
\]
we obtain
\begin{align*}
R(\ph) &= 2\big((\ph')^{1/2}-1\big) = 2 \sqrt{1+f'}-2 \\
&= f' - \frac 12 \int_0^1 \frac {1-t}{(1+tf')^{\frac 32}}f'\, dt\, f'  
 \\&=: f'+F(f')f'   ,
\end{align*}
with $F \in C^{\om}(\R_{>-1},\R)$ satisfying $F(0)=0$. 

Because $\ph = \on{Id} + f$ is a diffeomorphism, we have $f'>-1$ and since $f'\in \A(\R)$
implies that $f'$ vanishes at $-\infty$ and at $\infty$ we can can even conclude that $f'>-1+\ep$ for some $\ep>0$. 
Therefore, $F(f')$ is a bounded function for each $f'\in \Diff_{\AA}(\R)$. Using that all the spaces $\A(\R)$ are $\mathcal B(\R)$-modules we conclude  
that $F(f')f'$ and, hence, $R(\ph)$ are elements of $\A(\R)$.

To check that the mapping $R: \Diff_{\AA}(\mathbb R)\to \A\big(\R,\mathbb R_{>-2}\big)$ 
is bijective, we use the identity $\frac14(\ga(x)+2)^2-1=f'(x)$ with $\ga=R(\ph)=R(\on{Id}+f)$. Using this identity it is straightforward to calculate that
$$R\circ R\i=\on{Id}_{\A},\qquad R\i\circ R=\on{Id}_{\on{Diff}} .$$

To compute the pullback of the $L^2$-metric via the $R$-map we first need to calculate its tangent mapping. To do this, let $h=X\circ\ph \in T_{\ph}\Diff_{\AA}(\R)$ and let $t\mapsto \ps(t)$ be a smooth curve in 
$\Diff_{\AA}(\mathbb R)$ with $\ps(0)=\on{Id}$ and $\p_t|_0 \ps(t) = X$. 
We have:
\begin{align*}
T_\ph R.h &= \p_t|_0 R(\ps(t)\o\ph) 
= \p_t|_0 2\Big(((\ps(t)\o\ph)_x)^{1/2} -1\Big)
= \p_t|_0 2((\ps(t)_x\o\ph)\,\ph_x)^{1/2}
\\&
= 2(\ph_x)^{1/2}\p_t|_0 ((\ps(t)_x)^{1/2}\o\ph)
=  (\ph_x)^{1/2}\big(\frac{\ps_{tx}(0)}{(\ps(0)_x)^{-1/2}}\o\ph\Big)
\\& 
= (\ph_x)^{1/2}(X'\o\ph) 
 = (\ph')^{1/2}(X'\o\ph)    .
\end{align*}
Using this formula we have for $h=X_1\circ\ph,k= X_2\circ\ph$: 
\begin{equation*}
 R^*\langle h ,k \rangle_{L^2} =  \langle T_\ph R.h ,T_\ph R.k\rangle_{L^2}= \int_\R X_1'(x)X_2'(x)\,dx =  \langle h,k\rangle_{\dot H^1}.\quad\qedhere
\end{equation*}
\end{proof}

\begin{cor*}
\label{geod_dist}
Given $\ph_0,\ph_1\in \Diff_\AA(\R)$ the geodesic $\ph(t,x)$ connecting them is given by
\begin{align}\label{eq:geod}
\ph(t,x)=R\i\Big((1-t)R(\ph_0)+tR(\ph_1) \Big)(x) \,
\end{align}
and their geodesic distance is
\begin{align}
d(\ph_0,\ph_1)^2=4\int_{\R} \big((\ph'_1)^{1/2}-(\ph'_0)^{1/2}\big)^2\;dx .
\end{align}

Furthermore, the support of the geodesic is localized in the following sense: if $\ph(t,x)= x + 
f(t,x)$ with $f(t) \in \AA(\R)$ and similarly for $\ph_0,\ph_1$, then $\on{supp}(\p_x f(t))$ is 
contained in $\on{supp}(\p_x f_0)\cup \on{supp}(\p_x f_1)$. 
\end{cor*}
\begin{proof}
The formula for the geodesic $\ph(t,x)$ is clear. The geodesic distance between $\ph_0,\ph_1$ is given as the $L^2$ difference  
between their $R$-maps:
\begin{align*}
d^{\Diff}(\ph_0,\ph_1)&=d^{\A}(R(\ph_0),R(\ph_1))= 
\int_0^1 \sqrt{\int_\R \big(R(\ph_1)-R(\ph_0)\big)^2\;dx}\;dt
\\&= 2 \sqrt{\int_\R \big((\ph'_1)^{1/2}-(\ph'_0)^{1/2}\big)^2\;dx}.
\end{align*} 

To prove the statement regarding the support of the geodesic, we use the inversion formula of $R$ to obtain
$$f'(t,x) = \ph'(t,x)-1=\frac14 \ga(t,x)(\ga(t,x)+4),$$
where $\ga(t,x) = (1-t)R(\ph_0)(x) + tR(\ph_1)(x)$ is the image of the geodesic under the $R$-map. Next we note that at the points where $f_i'(x) = 0$ we have $\ph'_i(x) = 1$ and $R(\ph_i)(x)=0$. Hence, at the points where both $f_0'(x)=f_1'(x)=0$ we also have $f'(t,x)=0$ for all $t\in[0,1]$.
\end{proof}
\begin{figure}
\begin{center}
\includegraphics[width=.49\textwidth]{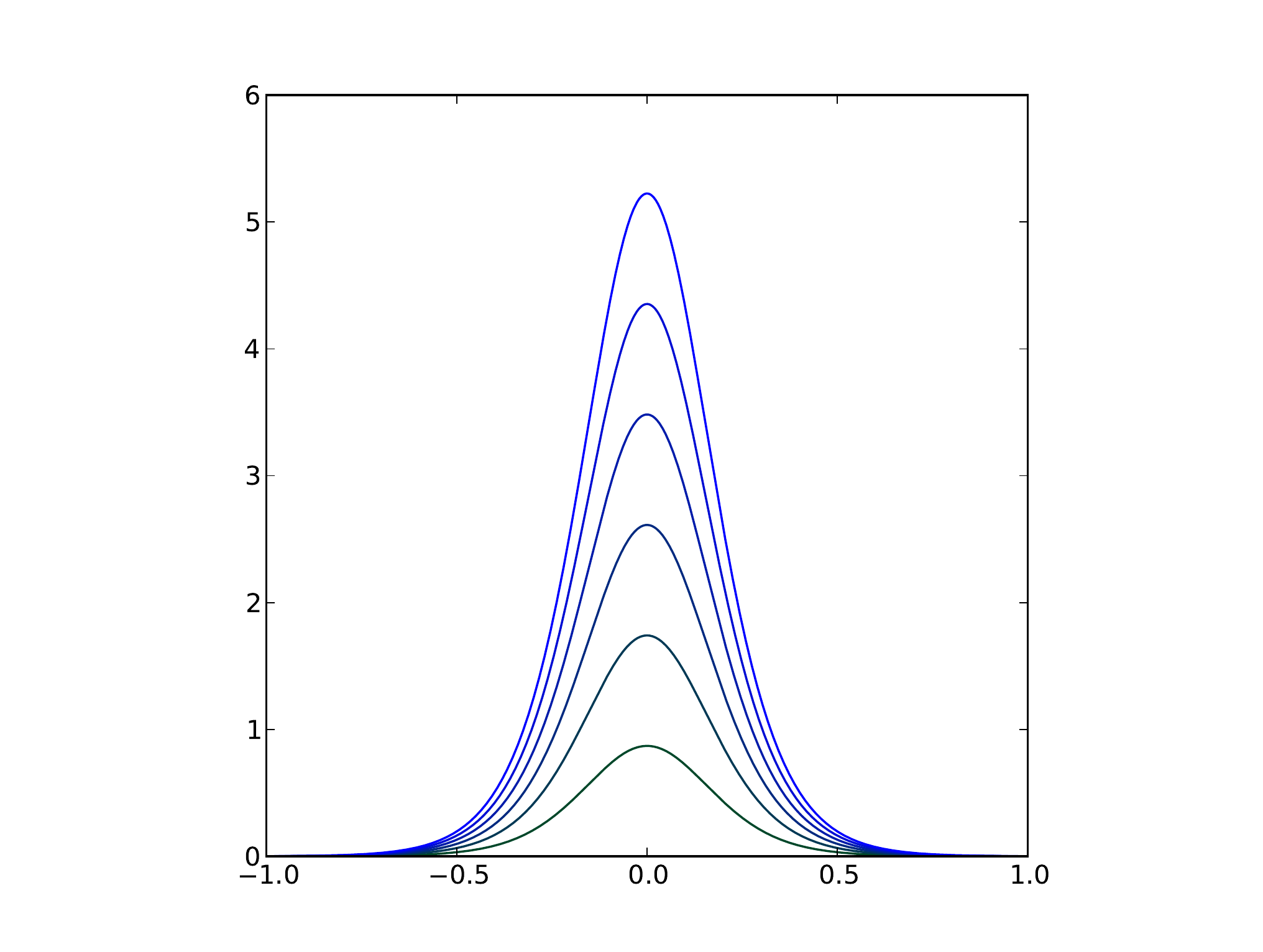}\hspace{-1cm}
\includegraphics[width=.49\textwidth]{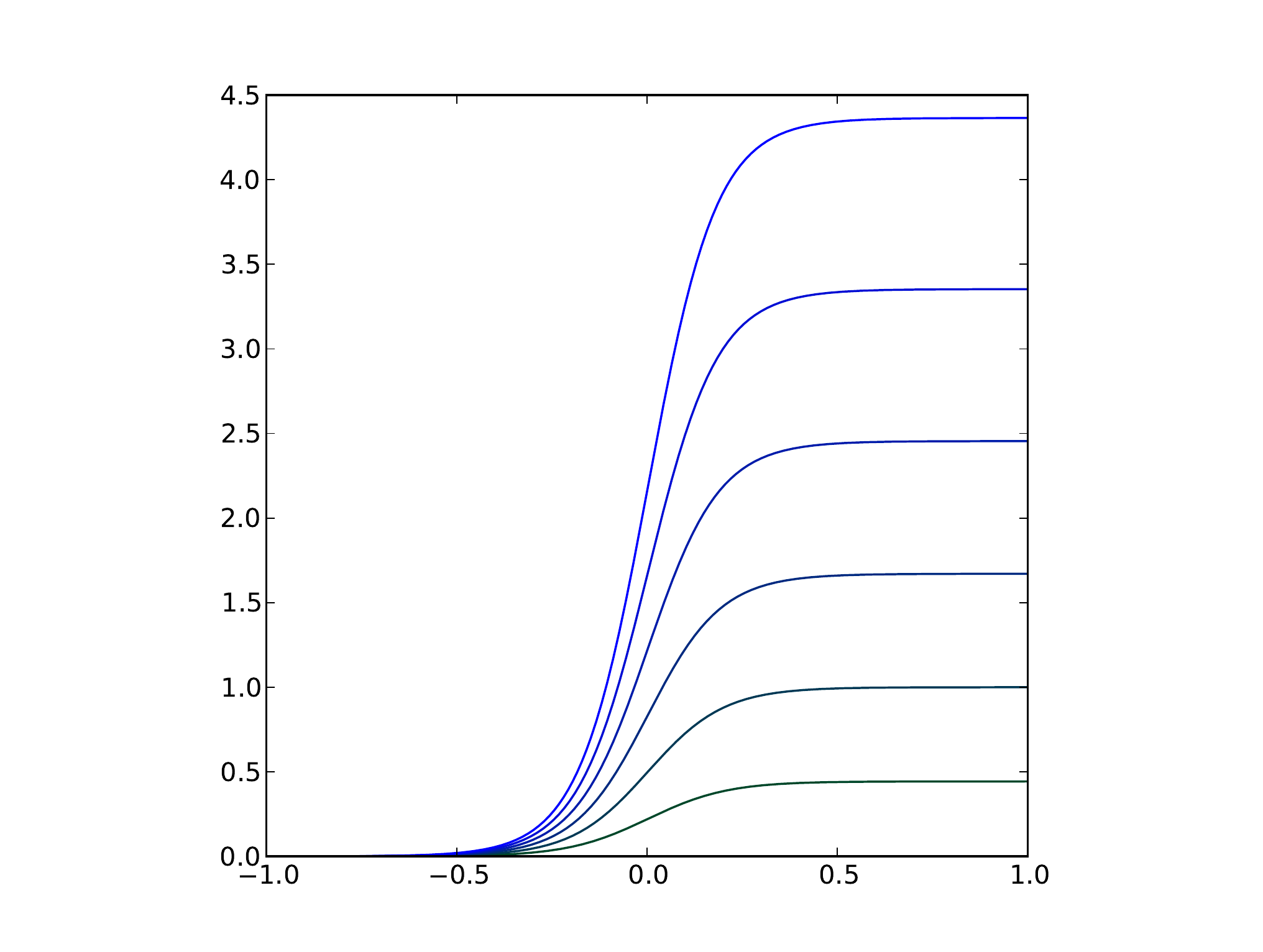}
\end{center}
\caption{A complete geodesic $\ph(t,x)\in\Diff_{\AA}(\R)$ sampled at time points 
$t=0,\tfrac12,1,\tfrac32,2,\tfrac52,3$. Left image: The geodesic in the $R$-map space. Right image: 
The geodesic in the original space, visualized as $\ph(t,x)-x$.}  
\label{fig:geod1}
\end{figure}

\subsection*{Example}
A geodesic connecting the identity to the diffeomorphism $\ph_1(x)= x + e^{-1/(x + 1)^2 }e^{- 1/(x - 1)^2}$
can be seen in Fig.~\ref{fig:geod1}.
In all the examples presented in this article, we  consider diffeomorphisms $\ph$ with $\on{supp}(\ph'-1)\subset[-1,1]$. 
We approximated the diffeomorphisms with $1000$ points on this interval. 
In the following lemma it is shown
that this behaviour does not hold in general.

\begin{lem*}
The metric space $\big(\Diff_{\AA}(\R),\dot H^1\big)$ is path-connected and geodesically convex but 
not geodesically complete.  

In particular, for every $\ph_0 \in \on{Diff}_\AA(\R)$ and $h \in T_{\ph_0}\on{Diff}_\AA(\R)$, 
$h\neq 0$ there exists a time $T\in\R$ such that $\ph(t,\cdot)$ is a geodesic for $|t|<|T|$ 
starting at $\ph_0$ with $\ph_t(0) = h$, but $\ph_x(T,x) = 0$ for some $x \in \R$.  
\end{lem*}

\begin{proof}
Set $\ga_0 = R(\ph_0)$ and $k = T_{\ph_0} R.h$. Then $R$ maps the geodesic $\ph(t)$ to
\[
R(\ph(t))(x) = \ga(t,x) = \ga_0(x) + t k(x)
\]
and the geodesic $\ph(t)$ ceases to exist when $\ga(t)$ leaves the image of the $R$-map, i.e. 
when $\ga(t,x) = -2$ for some pair $(t,x)$. Consider the function $g(x)= |2 + \ga_0(x)| / |k(x)|$ 
and set $g(x) = \infty$ where $k(x)=0$. Since we assumed $h\neq 0$, there is at least one $x\in\R$, 
such that $g(x)$ is finite. Since $k(x) \to 0$ as $|x| \to \infty$ we have $g(x) \to \infty$ for 
$x$ large, and $g$ attains the minimum at some point. Let this point be $x_0$ and define 
$T=-(2+\ga_0(x_0)) / k(x_0)$. Then $|T|$ is the time when $\ga(t,x)$ first reaches $-2$. So for 
$|t| < |T|$ the geodesic $\ga(t)$ lies in $\A(\R,\R_{>-2})$ and $\ga(T, x_0)=-2$. Then we have      
\[
\ph_x(T,x_0)=1+\frac14 \ga(T,x_0) (\ga(T,x_0)+4) = 1-1=0,
\]
as required. This proves that the space is not geodesically complete. 

The statement regarding path connectedness and geodesic convexity are direct consequences of the 
path connectedness and  convexity of $\A(\R,\R_{>-2})$. 
\end{proof}
\begin{figure}
\begin{center}
\includegraphics[width=.49\textwidth]{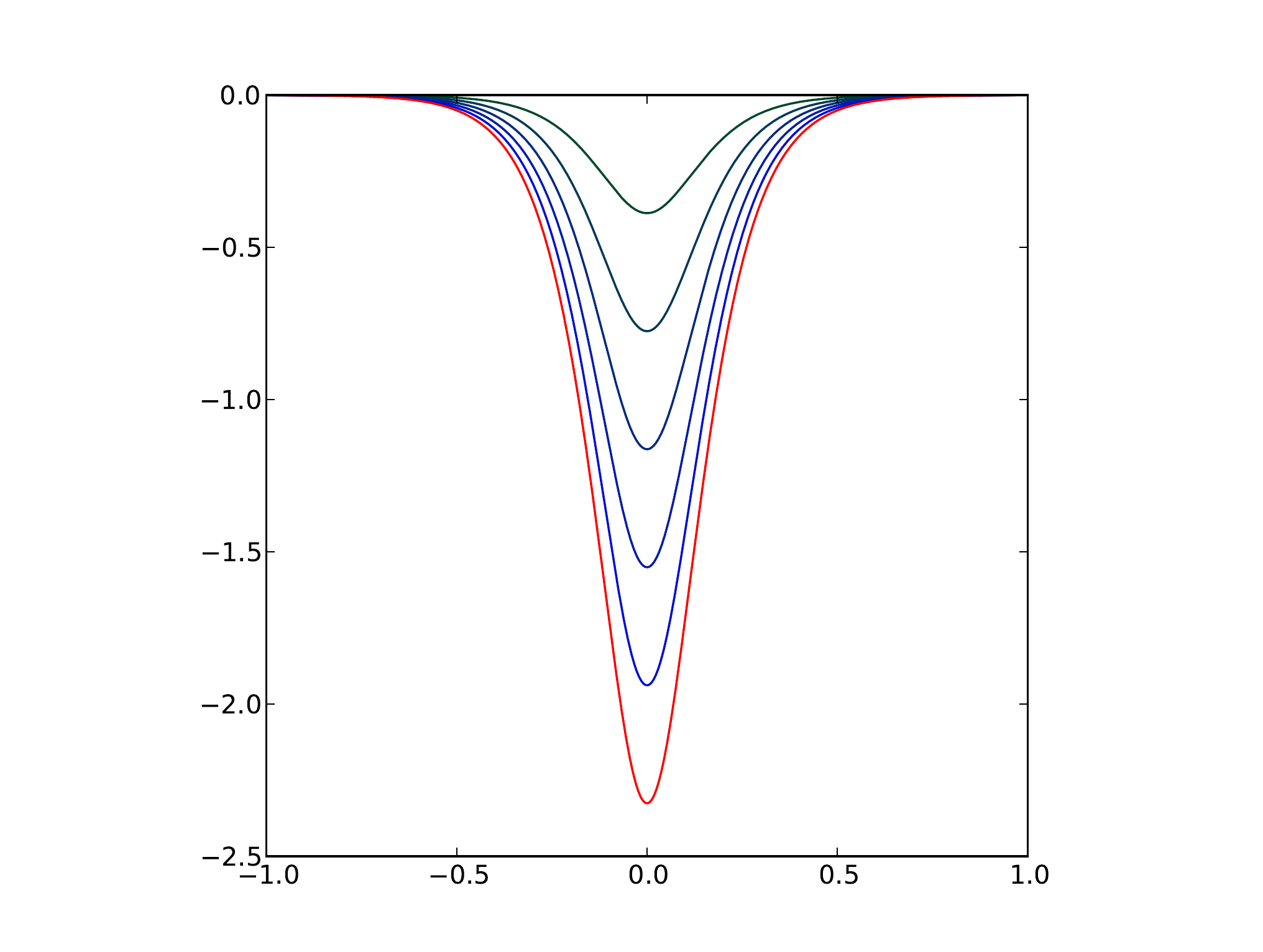}\hspace{-1cm}
\includegraphics[width=.49\textwidth]{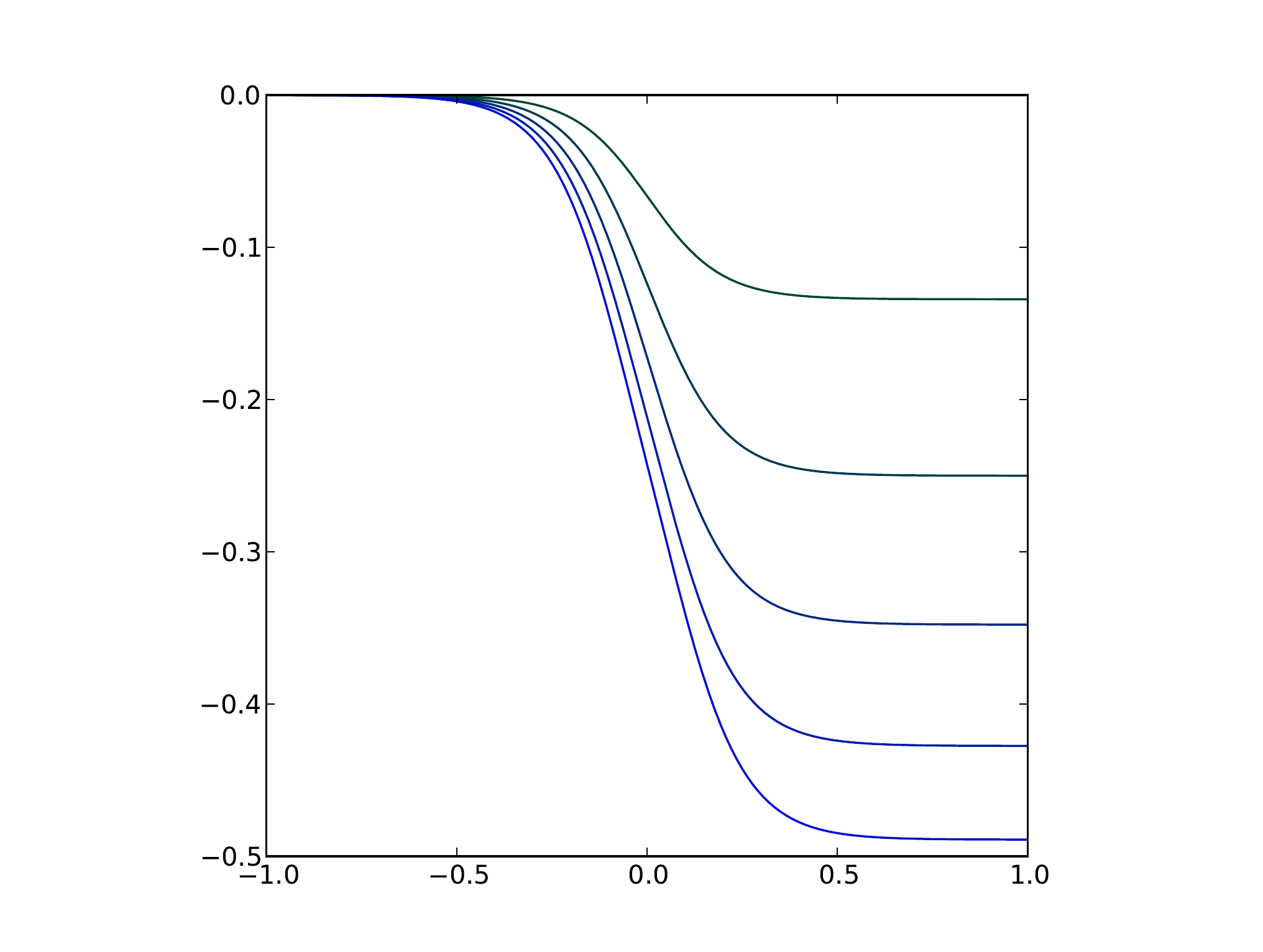}
\end{center}
\caption{An incomplete geodesic $\ph(t,x)\in\Diff_{\AA}(\R)$. Left image: The geodesic in $R$-map space at time points $t=0,\tfrac12,1,\tfrac32,2,\tfrac52,3$. At time $t=3$ (red line) the geodesic has already left the space of $R$-maps.
Right image: The geodesic in original space visualized as $\ph(t,x)-x$ sampled at time points $t=0,\tfrac12,1,\tfrac32,2,\tfrac52$. }
\label{fig:incomplete}
\end{figure}

\subsection*{Example}
This behaviour is illustrated in Fig.~\ref{fig:incomplete}, where we have again chosen $\ph_0=\on{Id}$ and we have solved the geodesic equation in the direction $$h(x)= R\i\left(\frac{-1}{4+4e^{-10x}}\right)$$
until the geodesic leaves the space of diffeomorphisms --- which happens approximately at time $t=2.58$.

\subsection{The Square-Root Representation on $\Diff_{\A}(\R)$.}
We will now study the homogeneous $H^1$-metric on diffeomorphism groups which do not allow a shift towards infinity. We can still use the same square-root representation 
as in the previous section, but now the image of this map will be a splitting submanifold in the image space $\A(\R,\R_{>-2})$.

\begin{thm*}
The square-root representation on the diffeomorphism group $\Diff_{\A}(\R)$ is a bijective mapping given by
$$ R:\left\{
\begin{aligned}
 \Diff_{\A}(\mathbb R)&\to \big(\on{Im}(R),\|\cdot\|_{L^2}\big)\subset
\big(\A\big(\R,\mathbb R_{>-2}\big),\|\cdot\|_{L^2}\big)\\
\ph &\mapsto
2\;\big((\ph')^{1/2}-1\big) .
\end{aligned}\right.
$$
The pullback of the restriction of the flat $L^2$-metric to $\on{Im}(R)$ via $R$ is again the homogeneous Sobolev metric of order one.
The image of the $R$-map is the splitting submanifold (in the sense of  \cite[Sect. 
27.11]{Michor1997}) of $\A(\mathbb R,\mathbb R_{>-2})$ given by:
$$\on{Im}(R) = \Big\{\ga \in \A(\R,\mathbb R_{>-2}):F(\ga):=\int_{\R}\ga\big(\ga+4\big)\; dx = 0 \Big\} .$$
\end{thm*}

\begin{proof}
The statement regarding the image of $R$, follows from the fact that a diffeomorphism 
$\ph=\on{Id}+g\in\Diff_{\AA}(\R)$  is also an element of $\Diff_{\A}(R)$
iff 
$$\int_\R (\ph'(x)-1)\;dx=\int_\R g'(x) \;dx =0.$$
Using that for $\ga=R(\ph)$ we have $\frac14(\ga(x)+2)^2-1=g'(x)$ we obtain the desired result.
The mapping 
$\A(\mathbb R,\mathbb R_{>-2})\ni f\mapsto 2((f+1)^{1/2}-1)\in \A(\mathbb R,\mathbb R_{>-2})$
is a diffeomorphism, and it maps $\A(\mathbb R,\mathbb R_{>-1})\cap \{f:\int f\,dx =0\}$ 
diffeomorphically onto $\on{Im}(R)$, which is therefore a splitting submanifold.
\end{proof}

\subsection*{Remark}
Note that we have $dF(\ga)(\de)= \int_{\mathbb R} (2\ga+4).\de \,dx$, but $2\ga+4$ is not in 
$\A$, only in the dual $\A^\ast$. So the so-called normal field along the codimension~1 submanifold 
$\on{Im}(\mathbb R)$ has no length; methods like the Gauss formula or the Gauss equation do not 
make sense. This is a diffeomorphic translation of the nonexistence of the geodesic equation in 
$\Diff_{\A}(\mathbb R)$.

\subsection{Geodesic Distance on $\on{Diff}_\A(\R)$} We have seen that on the space $\on{Diff}_\A(\R)$ the geodesic equation does not exist. It is, however, possible to define the geodesic distance between two diffeomorphisms $\ph_0, \ph_1 \in \on{Diff}_\A(\R)$ as the infimum over all paths $\ph:[0,1] \to \on{Diff}_\A(\R)$ connecting these:
\[
d^\A(\ph_0,\ph_1) = \inf_{\substack{\ph(0)=\ph_0 \\ \ph(1)=\ph_1}} \int_0^1 \sqrt{ G_{\ph(t)}(\p_t \ph(t), \p_t \ph(t))}\; dt.
\]
In Corollary \ref{geod_dist} we gave an explicit formula for the geodesic distance $d^\AA$ on the space $\on{Diff}_\AA(\R)$. It turns out that the geodesic distance $d^\A$ on $\on{Diff}_\A(\R)$ is the same as the restriction of $d^\AA$ to $\on{Diff}_\A(\R)$. Note that this is not a trivial statement, since in the definition of $d^\A$ the infimum is taken over all paths lying in the smaller space $\on{Diff}_\A(\R)$.

\begin{thm*}
The geodesic distance $d^\A$ on $\on{Diff}_\A(\R)$  coincides with the restriction of $d^\AA$ to $\on{Diff}_\A(\R)$, i.e. for $\ph_0, \ph_1 \in \on{Diff}_\A(\R)$ we have
\[
d^\A(\ph_0, \ph_1) = d^\AA(\ph_0, \ph_1).
\]
\end{thm*}

\begin{proof}
Because the space $\on{Diff}_{\A}(\R)$ is smaller than $\on{Diff}_{\AA}(\R)$, we have the inequality $d_{\AA} \leq d_{\A}$. The following argument will establish the other inequality $d_{\A} \leq d_{\AA}$ and hence with both of these inequalities together we will be able to conclude that $d_{\AA} = d_{\A}$.

Take $\ph_0, \ph_1 \in \on{Diff}_\A(\R)$ and let $\ph(t)$ be a path connecting them in the larger space $\on{Diff}_\AA(\R)$ . Then $\ga(t) = R(\ph(t))$ is a path in $\A(\R, \R_{>-2})$ and its length is measured by
\[
L(\ga) = \int_0^1 \sqrt{\int_{\R} \p_t \ga(t,x)^2 dx} dt\,.
\]
We also have the functional
\[
F(\ga) = \int_{\R} \ga(x)^2 + 4\ga(x) dx\,,
\]
which measures whether $\ga$ lies in the image of $\on{Diff}_{\A}(\R)$ under the $R$-map. 

We will construct a sequence $\wt \ga_n$ of paths with $n\to \infty$, such that these paths satisfy $F(\wt\ga_n)=0$ and $L(\wt \ga_n) \to L(\wt\ga)$. This will show that each path in $\on{Diff}_\AA(\R)$ can be approximated arbitraily well by path in the smaller space $\on{Diff}_\A(\R)$ and hence we will have established the other inequality $d_{\A} \leq d_{\AA}$.

Since $\ga(t)$ and $\p_t \ga(t)$ decay to 0 as $x \to \infty$ there exists for each $n > 0$ some $x_n \in \R$ such that
\[
\begin{array}{c}
|\ga(t,x)| < \frac 1n \\
|\p_t\ga(t,x)| < \frac 1n
\end{array}\qquad
 x>x_n\,.
\]
We also define $\ep_n = \frac 1n$. Let $\ps:\R\to\R$ be a smooth function with support in $[-1,1]$, $\ps(x) \geq 0$ and $\int \ps(x) dx = 1$ as well as $\int \ps(x)^2 dx = 1$. Now define the new path
\[
\wt \ga_n(t,x) = \ga(t,x) + \al_n(t) \ep_n \ps\left(\ep_n(x-x_n - \ep_n\i)\right)
\]
with the function $\al_n(t)$ determined by $F(\wt\ga_n(t))=0$. To make the calculations easier, we set $\ps_n(x) = \ps\left(\ep_n(x-x_n - \ep_n\i)\right)$ and we note that
\begin{align*}
\ep_n \int_\R \ps_n(x) dx &= 1, &
\ep_n \int_\R \ps_n(x)^2 dx &= 1.
\end{align*}
Writing the condition $F(\wt \ga_n(t))=0$ more explicitly we obtain
\begin{align*}
F(\wt \ga_n(t)) &= \int_\R \ga(t)^2 + 2 \al_n(t) \ep_n \ps_n \ga(t) + \al_n(t)^2 \ep_n^2 \ps_n^2  + 4 \left(\ga(t) + \al_n(t) \ep_n \ps_n\right) dx \\
&= \ep_n \al_n(t)^2 + \left( 4 + 2 \ep_n \int_\R \ps_n \ga(t) dx\right) \al_n(t) + F(\ga(t))\,.
\end{align*}
We need to estimate the integral 
\[
C_n(t) := \ep_n \int_\R \ps_n(x) \ga(t, x) dx
\]
to be able to control $\al(t)$. We have
\begin{align*}
|C_n(t)| &= \left| \ep_n \int_\R \ps\left(\ep_n(x-x_n - \ep_n\i)\right) \ga(t, x) dx \right| \\
&\leq \ep_n \int_{x_n}^{x_n + 2\ep_n\i} \ps\left(\ep_n(x-x_n - \ep_n\i)\right) \left| \ga(t,x) \right| dx \\
&\leq \ep_n \int_{x_n}^{x_n + 2\ep_n\i} \ps\left(\ep_n(x-x_n - \ep_n\i)\right) \frac 1ndx \\
& \leq \frac 1n\,.
\end{align*}
Hence we see that for large $n$ we will have $\al_n(t) \to -F(\ga(t))/4$ and the convergence is 
uniform in $t$. To see that $\al(t)$ is smooth in $t$ we can use the explicit formula 
\[
\al_n(t) = -\ep_n\i\left(2+C_n(t)\right) - \sqrt{ \ep_n^{-2} \left(2+C_n(t)\right)^2 - \ep_n\i F(\ga(t))}
\]
and note that for $\ep_n$ sufficiently small the term under the square root will always be positive.

Now we look at the length of the path $\wt\ga_n$. Let us consider only the inner integral 
$\int_\R \p_t\wt\ga_n(t,x) dx$. We will show the convergence 
\[
\int_\R \left(\p_t\wt\ga_n(t,x)\right)^2 dx \to \int_\R \left(\p_t\ga(t,x)\right)^2 dx
\]
uniformly in $t$, which will imply the convergence $L(\wt\ga_n) \to L(\ga)$. To do so, we look at
\begin{align*}
\int_\R \left(\p_t\wt\ga_n(t)\right)^2 dx 
&= \int_\R \left(\p_t\ga(t)\right)^2 + 2\p_t\al_n(t) \ep_n \ps_n \p_t \ga(t) + \left(\p_t\al_n(t)\right)^2 \ep_n^2 \ps_n(x)^2 dx \\
&= \int_\R \left(\p_t\ga(t)\right)^2 dx + 2 \p_t\al_n(t) \ep_n \int_\R \ps_n \p_t \ga(t) dx + \left(\p_t \al_n(t)\right)^2\ep_n.
\end{align*}
The integral
\[
D_n(t) := \ep_n \int_\R \ps_n(x) \p_t\ga(t, x) dx
\]
can be estimated in the same way as $C_n(t)$ previously to obtain $D_n(t) \leq \tfrac 1n$. It remains 
to bound $\p_t \al(t)$ uniformly in $t$ to show convergence. We differentiate the equation 
\[
\ep_n \al_n(t)^2 + \left( 4 + 2 \ep_n \int_\R \ps_n \ga(t) dx\right) \al_n(t) + F(\ga(t)) = 0
\]
that defines $\al(t)$, which gives us
\[
\p_t\al_n(t) \left( 2\ep_n \al_n(t) + 4 + 2C_n(t)\right) + T_{\ga(t)} F.\p_t \ga(t) = 0
\]
and thus
\[
\p_t\al_n(t) = -\frac{T_{\ga(t)} F.\p_t \ga(t)}{4 + 2\ep_n \al_n(t) + 2C_n(t)}\,.
\]
We see that $\p_t \al_n(t) \to -T_{\ga(t)} F.\p_t \ga(t)/4$ and the convergence is uniform 
in $t$. Thus we have shown that convergence $L(\wt \ga_n) \to L(\ga)$ of the length functional. 
\end{proof}

\subsection{Submanifold $\on{Diff}_\A(\R)$ Inside $\on{Diff}_\AA(\R)$}
The following theorem deals with the question how $\on{Diff}_\A(\R)$ lies inside the extension 
$\on{Diff}_\AA(\R)$. We give an upper bound for how often a geodesic in $\on{Diff}_\AA(\R)$ might 
intersect or be tangent to $\on{Diff}_\A(\R)$. It is only an upper bound because the geodesic 
might leave the group of diffeomorphisms before intersecting $\on{Diff}_\A(\R)$.   

\begin{thm*}
\label{lem_par}
Consider a geodesic $\ph(t)$ in $\on{Diff}_\AA(\R)$ starting at $\ph(0)=\ph_0$ with initial 
velocity $\p_t \ph(0) = u_0 \o \ph_0$ and denote by $u(t) = \p_t \ph(t) \o \ph(t)\i$ the 
right-trivialized velocity. Then the size of the shift at infinity is given by  
\begin{align*}
\on{Shift}(\ph(t))& = \on{Shift}(\ph_0) + t u_0(\infty) + \frac{t^2}4 \int_\R (u_0')^2 dx \\
u(t,\infty) &= u_0(\infty) + t \int_\R (u_0')^2 dx.
\end{align*}
This means that every geodesic in $\on{Diff}_\AA(\R)$ intersects $\on{Diff}_\A(\R)$ at most twice, 
and every geodesic is tangent to a right coset of $\on{Diff}_\A(\R)$ at most once. 

For $\ph_0,\ph_1\in\Diff_{\A}(\R)$ we can give the following formula for the size of the shift along the connecting minimal geodesic:
$$
\on{Shift}(\ph(t))=\frac{t^2-t}{4}\big\|R(\ph_0)-R(\ph_1)\big\|^2_{L^2} 
= (t^2-t)\big\|\sqrt{\ph_0'}-\sqrt{\ph_1'}\,\big\|^2_{L^2}   .
$$
\end{thm*}

\begin{proof}
To make the computations easier, define the following variables:
\begin{align*}
\ga_0 &= R(\ph_0) \\
k_0 &= T_{\ph_0} R.(u_0 \o \ph_0) = \sqrt{\ph_0'} u_0' \o \ph_0 \\
\ga(t) &= R(\ph(t)) = \ga_0 + tk_0.
\end{align*}
For a diffeomorphism $\ph \in \Diff_{\AA}(\R)$ the size of the shift at $\infty$ is given by
\[
\on{Shift}(\ph) = \lim_{x \to \infty} \ph(x) - x = \int_{-\infty}^\infty \ph'(x) - 1\;dx = \frac 14 \int_{-\infty}^\infty R(\ph)^2 + 4R(\ph)\;dx.
\]
Similarly, the value of a function $u \in \AA$ at $\infty$ can be computed by
\begin{align*}
u(\infty) &= \int_{-\infty}^\infty u'(x)dx = \int_{-\infty}^\infty (u'\o\ph)(x) \ph'(x)\; dx \\
&= \frac 12 \int_{-\infty}^\infty (R(\ph)+2) T_\ph R.(u\o\ph)(x)\; dx.
\end{align*}
Note that this holds for any $\ph \in \Diff_{\A_1}(\R)$. Since the $R$-map pulls back the $L^2$-metric to the $\dot H^1$-metric we also have the identity
\[
\int_{-\infty}^\infty u'(x)dx = \int_{-\infty}^\infty (T_\ph R.(u\o\ph))^2 dx.
\]
For the sake of convenience let us rewrite the last two equations using the variables $u_0, \ga_0$ and $k_0$.
\begin{align*}
u_0(\infty) &= \frac 12 \int_{-\infty}^\infty (\ga_0 + 2) k_0 \;dx &
\int_{-\infty}^\infty (u_0')^2\; dx = \int_{-\infty}^\infty k_0^2\; dx
\end{align*}
Now computing the shift of $\ph(t)$ at $\infty$ is easy:
\begin{align*}
\on{Shift}(\ph(t)) &= \frac 14 \int_{-\infty}^\infty (\ga_0 + tk_0)^2 + 4(\ga_0 + tk_0)\;dx \\
&= \frac 14 \int_{-\infty}^\infty \ga_0^2 + 4\ga_0 + 2t \ga_0 k_0 + 4t k_0 + t^2 k_0^2 \;dx\\
& = \on{Shift}(\ph_0) + t u_0(\infty) + \frac{t^2}4 \int_{-\infty}^\infty (u_0')^2 dx.
\end{align*}
Computing the value of $u(t)$ at $\infty$ is just as simple:
\begin{align*}
u(t, \infty) &= \frac 12 \int_{-\infty}^\infty (\ga_0 + tk_0 + 2) k_0 \;dx \\
&= u_0(\infty) + t \int_\R (u_0')^2 dx.
\end{align*}

If we start with $\ph_0,\ph_1\in\Diff_{\A}(\R)$, then the geodesic connecting them has $k_0 = \ga_1 - \ga_0$ with $\ga_1 = R(\ph_1)$. Some algebraic manipulations, keeping in mind that $\on{Shift}(\ph_i) = \frac 14 \int \ga_i^2 + 4\ga_i = 0$, give us
\begin{align*}
u_0(\infty) &= \frac 12 \int_{-\infty}^\infty (\ga_0 + 2)(\ga_1 - \ga_0)\;dx = \frac 12 \int_{-\infty}^\infty \ga_0 \ga_1 - \ga_0^2 - 2\ga_0 + 2\ga_1 \;dx \\
&= \frac 12 \int_{-\infty}^\infty \ga_0 \ga_1 - \ga_0^2 + \frac 12 \ga_0^2 -\frac 12 \ga_1^2 \;dx = -\frac 14 \int_{-\infty}^\infty (\ga_1 - \ga_0)^2\; dx,
\end{align*}
which in turn leads to
\[
\on{Shift}(\ph(t)) = tu_0(\infty) + \frac{t^2}4 \int_{-\infty}^\infty (u_0')^2\; dx = \frac{t^2-t}4 \int_{-\infty}^\infty (\ga_1 - \ga_0)^2\; dx.
\]
This completes the proof.
\end{proof}

An example of a geodesic illustrating the behaviour described in the lemma can be seen in Fig.~\ref{fig:paraboloid}. 

\begin{figure}
\begin{center}
\includegraphics[width=.49\textwidth]{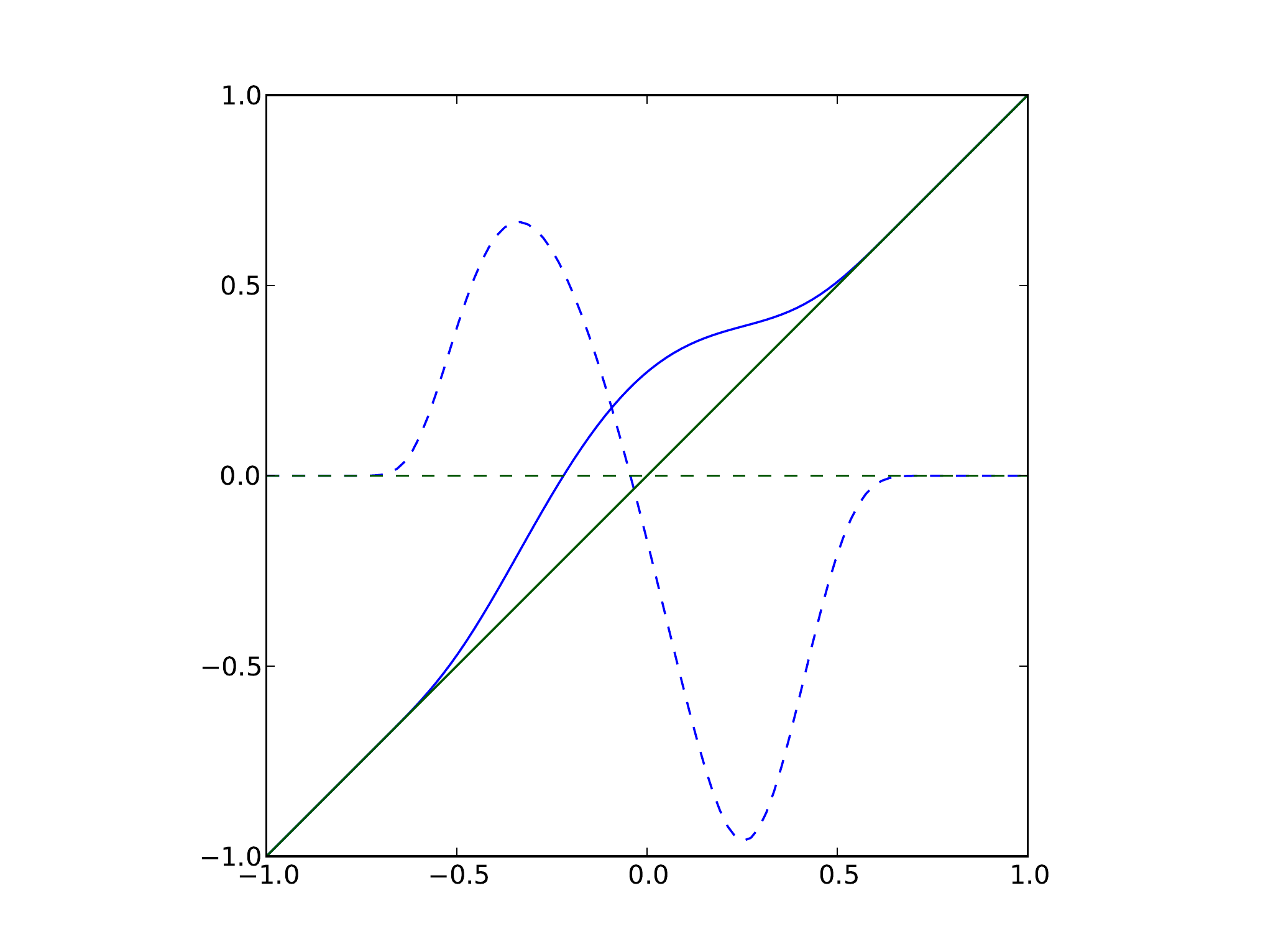}\hspace{-1cm}
\includegraphics[width=.49\textwidth]{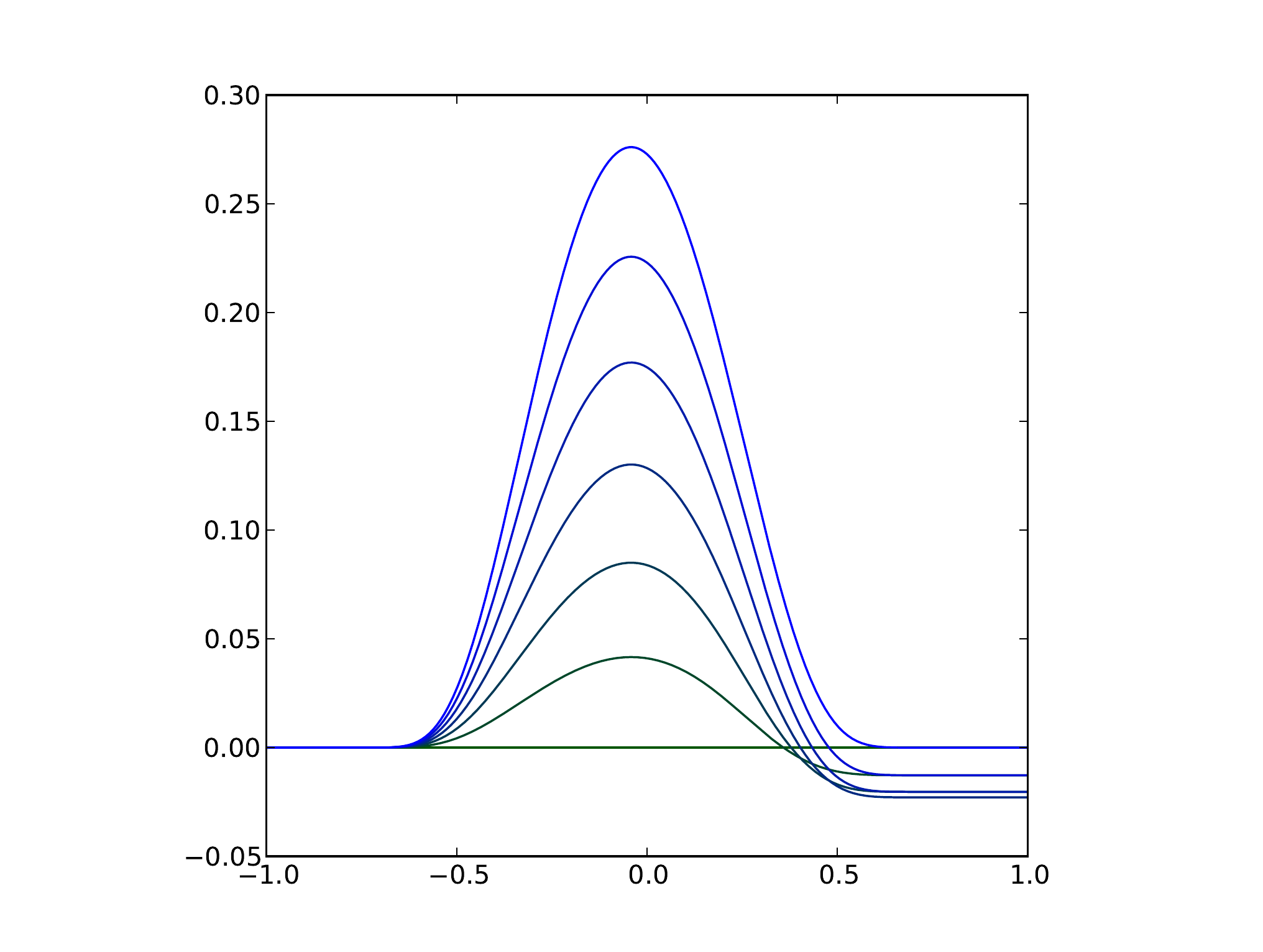}
\end{center}
\caption{Geodesic $\ph(t,x)\in\Diff_{\AA}(\R)$ between two diffeomorphisms in $\Diff_{\A}(\R)$ sampled at times $t=0,\tfrac16,\tfrac26,\tfrac36,\tfrac46,\tfrac56,1$.
 Left image: The geodesic in $R$-map space. Right image: The geodesic in original space, visualized as $\ph(t,x)-x$.}
\label{fig:paraboloid}
\end{figure}

\subsection{Solving the Hunter--Saxton Equation}\label{SolvingHS}
The theory described in the preceding sections  allows us to construct an analytic solution formula for the HS equation on $\AA(\R)$.
Here  $\AA(\R)$ denotes one of the function spaces $C_{c,1}^{\infty}(\R)$, $\mathcal{S}_1(\R)$ or $W^{\infty,1}_1(\R)$, as defined in Sects.~\ref{diffc}--\ref{diffw}.
\begin{thm*}[Solutions to the HS equation]
Given an initial value $u_0$ in $\AA(\R)$  the solution to the HS equation is given by
$$u(t,x) = \ph_t(t,\ph\i(t,x)),\quad\text{with}\quad \ph(t,x)=R\i\big(t u'_0\big)(x).$$
In particular, this means that a solution with initial condition in one of the spaces $C_{c}^{\infty}(\R)$, $\mathcal{S}(\R)$ or $W^{\infty,1}(\R)$ exists for all time $t>0$, if and only if $u'_0(x)\geq0$ for all $x\in\R$.
All solutions are real-analytic in time in the sense of \cite[Section 9]{Michor1997}.
\end{thm*}

\begin{proof}
By the theory of the previous sections, we know that the path $$\ph(t,x)=R\i(t \ga)(x),$$ with $\ph(0,x) = x$ is a 
solution to the geodesic equation for every $\ga\in\AA(\R)$. 
It remains to choose $\ga$ such that the initial condition
$$\ph_t(0,\ph\i(0,x))=\ph_t(0,x)=u_0(x)$$
is  satisfied. This can be achieved by choosing $\ga=T_{\on{Id}}R.u_0$, since
\begin{align*}
\ga=\p_t|_{0}(t\ga)=\p_t|_0 R(\ph(t))=T_{\ph(t)}R.\ph_t(t)\big|_{0}=T_{\on{Id}}R. u_0.
\end{align*}
Using the formula $T_\ph R.h = (\ph')^{-1/2} h' $ from the proof of Theorem~\ref{RmapDiffAA} yields
$T_{\on{Id}}R.u_0=u_0'$.

The solution is real-analytic in $t$ since $t\mapsto t.u_0$ is a real-analytic curve in 
$\AA(\mathbb R)$ and since $R\i$ respects real-analytic curves; see~\cite[Section 9]{Michor1997}.

Given $x_0, u_0$ such that $u_0'(x_0)<0$ there exist $t_0\in\R$ with $t_0 u_0'(x_0)<-2$. 
Thus the geodesic at time $t_0$ has left the $R$-map space, and the solution of the HS equation leaves the space $\AA$.
\end{proof}

A more explicit formula for the solution is given by
\begin{align*}
u(t,x) &= u_0(\ph\i(t,x)) + \frac t2  \int_{-\infty}^{\ph\i(t,x)} u_0'(y)^2\;dy \\
\ph(t,x) &= x + \frac 14 \int_{-\infty}^x t^2 (u_0'(y))^2 + 4tu_0'(y)\;dy.
\end{align*}

\subsection*{Remark} The HS equation on the real line also provides an example of how geometry and PDE behaviour influence each other. It was shown in Sect.~\ref{H1metric} that the geodesic equation on $\on{Diff}_\A(\R)$ does not exist because the condition $\on{ad}^\ast(u)\,u \in \check G_{\on{Id}}(u)$ is not satisfied. From a naive point of view we could start with the energy
\begin{equation}
\label{eq:hs_en}
E(u)= \int_0^1 \int_{-\infty}^\infty u_x^2\; dx\; dt
\end{equation}
defined on functions $u \in C^\infty([0,1], \A(\R))$ and take variations of the form $\de u = \et_t + \et_x u - \et u_x$ with fixed endpoints $\et(0,x) = \ph_0(x)$ and $\et(1,x) = \ph_1(x)$. This would lead, after some integration by parts, to
\[
\langle DE(u), \de u \rangle = \int_0^1 \int_{-\infty}^\infty \left( u_{txx} + (uu_x)_{xx} - u_xu_{xx}\right) \et \;dx\; dt
\]
and we could now declare
\[
u_{txx} + (uu_x)_{xx} - u_xu_{xx} = 0
\]
to be the geodesic equation. It is in any case the equation which the critical points of the energy functional \eqref{eq:hs_en} must satisfy. But this equation has no solutions in $\A$. It is shown in Theorem~\ref{lem_par} that a solution $u \in C^\infty([0,1], \AA(\R))$ intersects $\A(\R)$ at most once. To find solutions, we must enlarge the space to $\AA(\R)$ and then the preceding theorem via the $R$-map gives us the existence of solutions.

\subsection{Continuing Geodesics Beyond the Group, or How Solutions of the Hunter--Saxton 
Equation Blow Up}
Consider a straight line $\ga(t) = \ga_0 + t\ga_1$ in $\A(\mathbb R,\mathbb R)$. 
Then $\ga(t)\in\A(\mathbb R,\mathbb R_{>-2})$ precisely for $t$ in an open interval $(t_0,t_1)$ 
which is finite at least on one side, say, at $t_1<\infty$. 
Note that 
$$
\ph(t)(x):=R\i(\ga(t))(x) = x+ \frac14\int_{-\infty}^x \ga^2(t)(u)+4\ga(t)(u)\,du
$$
makes sense for all $t$,
that $\ph(t):\mathbb R\to \mathbb R$ is smooth and that $\ph(t)'(x)\ge0$ for all $x$ and $t$; thus,
$\ph(t)$ is monotone non-decreasing. Moreover, $\ph(t)$ is proper and surjective since $\ga(t)$
vanishes at $-\infty$ and $\infty$. Let 
$$
\on{Mon}_{\AA}(\mathbb R) := \big\{\on{Id}+f: f\in\AA(\mathbb R,\mathbb R), f'\geq-1\big\}
$$
be the monoid (under composition) of all such functions.

For $\ga\in\A(\mathbb R,\mathbb R)$ let 
$x(\ga):=\min\{x\in \mathbb R\cup\{\infty\}: \ga(x)=-2\}.$
Then for the line $\ga(t)$ from above we see that $x(\ga(t))<\infty$ for all $t>t_1$.
Thus, if the `geodesic' $\ph(t)$ leaves the diffeomorphism group at $t_1$, it never comes back but 
stays inside $\on{Mon}_{\AA}(\mathbb R)\setminus \Diff_\AA(\mathbb R)$ for the rest of its life. 
In this sense, $\on{Mon}_{\AA}(\mathbb R)$ is a \emph{geodesic completion} of 
$\Diff_{\AA}(\mathbb R)$, and $\on{Mon}_\AA(\mathbb R)\setminus \Diff_\AA(\mathbb R)$ is the 
\emph{boundary}.

What happens to  the corresponding solution $u(t,x)= \ph_t(t,\ph(t)\i(x))$ of the HS equation? 
In certain points it has infinite derivative, it may be multivalued, or its graph can contain 
whole vertical intervals. If we replace an element $\ph\in\on{Mon}_\AA(\mathbb R)$ by its graph 
$\{(x,\ph(x)):x\in \mathbb R\}\subset \mathbb R$ we get a smooth `monotone' submanifold, a smooth 
monotone relation. The inverse $\ph\i$  is then also a smooth monotone relation.
Then $t \mapsto\{(x,u(t,x)):x\in \mathbb R\}$ is a (smooth) curve of relations. 
Checking that it satisfies the HS equation is an exercise left for the interested reader. 
What we have described here is the \emph{flow completion} of the HS equation in the spirit of
\cite{KhesinMichor2004}.

\subsection{Soliton-Like Solutions} For a right-invariant metric on a diffeomorphism group one can 
ask whether (generalized) solutions $u(t)=\ph_t(t)\o\ph(t)\i$ exist such that the momenta 
$\check G(u(t))=:p(t)$ 
are distributions with finite support. 
Here the geodesic $\ph(t)$ may exist only in some suitable Sobolev completion of the diffeomorphism 
group. By the general theory (see in particular \cite[Sect.~4.4]{Michor2006a}, or 
\cite{Micheli2012_preprint}) the momentum
$\on{Ad}(\ph(t))^*p(t) = \ph(t)^*p(t) = p(0)$ is constant. In other words,
$p(t)=(\ph(t)\i)^*p(0)=\ph(t)_*p(0)$, i.e., the momentum is carried forward by the flow and 
remains in the space of distributions with finite support.
The infinitesimal version (take $\p_t$ of the last expression) is 
$p_t(t) = -\L_{u(t)}p(t)= -\on{ad}_{u(t)}{}^*\,p(t)$; c.f. Sect.~\ref{adjoint}.
The space of $N$-solitons of order 0 consists of momenta of the form $p_{y,a} = \sum_{i=1}^N a_i \de_{y_i}$ with $(y,a) \in \R^{2N}$. Consider an initial soliton $p_0=\check G(u_0) = - u_0'' = \sum_{i=1}^N a_i \,\de_{y_i}$ with $y_1<y_2<\dots<y_N$. Let $H$ be the Heaviside function
\[
H(x) = \begin{cases}
0, &x<0, \\
\frac 12, &x = 0, \\
1, &x > 0,
\end{cases}
\]
and $D(x) = 0$ for $x\leq 0$ and $D(x)=x$ for $x > 0$. We will see later why the choice $H(0) = \tfrac 12$ is the most natural one; note that the behavior is called the Gibbs phenomenon. With these functions we can write
\begin{align*}
u_0''(x) &= -\sum_{i=1}^N a_i \de_{y_i}(x) \\
u_0'(x) &= -\sum_{i=1}^N a_i H(x-y_i) \\
u_0(x) &= -\sum_{i=1}^N a_i D(x-y_i).
\end{align*}
We will assume henceforth that $\sum_{i=1}^N a_i = 0$. Then $u_0(x)$ is constant for $x > y_N$ and thus $u_0 \in H_1^1(\R)$; with a slight abuse of notation we assume that $H_1^1(\R)$ is defined similarly to $H_1^\infty(\R)$. Defining $S_i = \sum_{j=1}^i a_j$ we can write
\[
u_0'(x) = -\sum_{i=1}^N S_i \left( H(x - y_i) - H(x - y_{i+1})\right).
\]
This formula will be useful because $\on{supp}( H(. - y_i) - H(. - y_{i+1})) = [y_i, y_{i+1}]$.

The evolution of the geodesic $u(t)$ with initial value $u(0) = u_0$ can be described by a system of ordinary differential equations (ODEs) for the variables $(y, a)$. We cite the following result.
\begin{thm*}[\cite{Holm2005}]
The map $(y,a) \mapsto \sum_{i=1}^N a_i \de_{y_i}$ is a Poisson map between the canonical symplectic structure on $\R^{2N}$ and the Lie--Poisson structure on the dual  $T^\ast_{\on{Id}} \on{Diff}_\A(\R)$ of the Lie algebra.
\end{thm*}

In particular, this means that the ODEs for $(y,a)$ are Hamilton's equations for the pullback Hamiltonian
\[
E(y,a) = \frac 12 G_{\on{Id}}(u_{(y,a)}, u_{(y,a)}),
\]
with $u_{(y,a)} = \check G\i(\sum_{i=1}^N a_i \de_{y_i}) = -\sum_{i=1}^N a_i D(. - y_i)$. We can obtain the more explicit expression
\begin{align*}
E(y,a) &= \frac 12 \int_\R \left(u_{(y,a)}(x)'\right)^2 \;dx 
= \frac 12 \int_\R \left(\sum_{i=1}^N S_i \mathbbm{1}_{[y_i, y_{i+1}]} \right)^2 \;dx\\
&= \frac 12 \sum_{i=1}^N S_i^2 (y_{i+1} - y_i) .
\end{align*}
Hamilton's equations $\dot y_i = \p E/\p a_i$, $\dot a_i = -\p E/\p y_i$ are in this case
\begin{align*}
\dot y_i(t) &= \sum_{j=i}^{N-1} S_i(t) (y_{i+1}(t) - y_i(t)), \\
\dot a_i(t) &= \frac 12 \left( S_i(t)^2 - S_{i-1}(t)^2\right).
\end{align*}

Using the $R$-map we can find explicit solutions for these equations as follows. 
Let us write $a_i(0) = a_i$ and $y_i(0) = y_i$. By Theorem \ref{SolvingHS} the geodesic with initial velocity $u_0$ is given by
\begin{align*}
\ph(t,x) &= x + \frac 14 \int_{-\infty}^x t^2 (u_0'(y))^2 + 4tu_0'(y)\;dy \\
u(t,x) &= u_0(\ph\i(t,x)) + \frac t2  \int_{-\infty}^{\ph\i(t,x)} u_0'(y)^2\;dy.
\end{align*}
First note that
\begin{align*}
\ph'(t,x) &= \left(1 + \frac t2 u_0'(x)\right)^2 \\
u'(t,z) &= \frac {u_0'\left(\ph\i(t,z)\right)}{1 + \frac t2 u_0'\left(\ph\i(t,z)\right)}.
\end{align*}
Using the identity $H(\ph\i(t,z) - y_i) = H(z - \ph(t,y_i))$ we obtain
\[
u_0'\left(\ph\i(t,z)\right) = -\sum_{i=1}^N a_i H\left(z - \ph(t, y_i)\right),
\]
and thus
\[
\left(u_0'\left(\ph\i(t,z)\right) \right)' = -\sum_{i=1}^N a_i \de_{\ph(t, y_i)}(z).
\]
Combining these we obtain
\begin{align*}
u''(t,z) &= \frac {1}{\left( 1+ \frac t2 u_0'\left(\ph\i(t,z)\right)\right)^2} \left( -\sum_{i=1}^N a_i \de_{\ph(t, y_i)}(z)\right) \\
&= \sum_{i=1}^N \frac{-a_i}{\left( 1 + \frac t2 u_0'(y_i)\right)^2} \de_{\ph(t, y_i)}(z).
\end{align*}
From here we can read off the solution of Hamilton's equations
\begin{align*}
y_i(t) &= \ph(t, y_i) \\
a_i(t) &= -a_i \left( 1 + \tfrac t2 u_0'(y_i)\right)^{-2}.
\end{align*}
When trying to evaluate $u_0'(y_i)$,
\[
u_0'(y_i) = a_i H(0) - S_i,
\] 
we see that $u_0'$ is discontinuous at $y_i$ and it is here that we seem to have the freedom to choose the value $H(0)$. However, it turns out that
we observe  the Gibbs phenomenon, i.e., only the choice $H(0) = \frac 12$ leads to solutions of 
Hamilton's equations. Also, the regularized theory of multiplications of distributions (see the 
discussion in \cite[Sect.~1.1]{GKOS01}) leads to this choice. 
Thus we obtain
\begin{align*}
y_i(t) &= y_i + \sum_{j=1}^{i-1} \left(\frac{t^2}4 S_j^2 - t S_j\right) \left(y_{j+1}-y_j\right) \\
a_i(t) &= \frac{-a_i}{\left( 1 + \frac t2 \left(\frac{a_i}2 - S_i \right)\right)^2} = -\left( \frac{S_i}{1 - \frac t2 S_i} - \frac{S_{i-1}}{1 - \frac t2 S_{i-1}}\right).
\end{align*}
It can be checked by direct computation that these functions indeed solve Hamilton's equations.

\section{Two-Component Hunter-Saxton Equation on Real Line}
\label{2hs}
In this section we will show, similarly to the previous section, that one can adapt the work of Lenells on the periodic two-component HS equation \cite{Lenells2011_preprint} to obtain results for the non-periodic case.
On the real line this system has been studied from an analytical viewpoint in Sect.~\cite[Section~4]{Wunsch2010b}.
\begin{thm*}
Let $\M=\Diff_{\A}(\mathbb R)\ltimes \A(\R,\R)$ and 
$\wt\M=\Diff_{\AA}(\mathbb R)\ltimes \A(\R,\R)$ be the semi-direct product Lie groups with multiplication
$$(\ph,\al)(\psi,\be)=(\ph\circ\psi, \be+\al\circ\psi),$$
where $\A$ and $\AA$ are as defined in \ref{setting}. Consider the following weak Riemannian metric $G$ on  $M$ and $\wt M$:
$$G_{(\on{Id},0)}((X,a),(Y,b)) = \int X'(x)Y'(x) + a(x)b(x)\;dx,$$
where
$(X,a)$ and $(Y,b)$ are elements of the corresponding Lie algebra.

Then the geodesic equation on $\wt\M$ is the two-component non-periodic HS equation given by
$$\boxed{\begin{aligned}
(u,\rho)&=(\ph_t\circ\ph\i,\al_t\circ\ph\i)\\
\qquad u_{t} &= -u u_x + \frac12 \int_{-\infty}^x   u_x(z)^2 +\rho^2(z)  \,dz\\
\rho_t &= -(\rho u)_x
\end{aligned}}$$
The geodesic equation does not exist on   $\M$, since the adjoint 
$\on{ad}((X,a))^*G_{(\on{Id},0)}$ is not in $G_{(\on{Id},0)}(\text{Lie algebra})$ for all $(X,a)$. 
These are not  robust Riemannian manifolds in the sense of \cite[Sect.~2.4]{Michor2012a_preprint}.
\end{thm*}

\subsection*{Remark}
Note that one obtains more so-called classical forms of the HS equation by  differentiating the  equation for $u_t$:
\begin{align*}
u_{tx} &= -u u_{xx}+\frac12 (-u_x^2+\rho^2);
\\
u_{txx} &= \big(-u u_{xx}+\frac12 (-u_x^2+\rho^2)\big)_x = -2 u_xu_{xx}-u u_{xxx}+\rho\rho_x .
\end{align*}
The second of the precedeing equations is the version which Lenells called the two-component HS equation in \cite{Lenells2011_preprint}.

\begin{proof}
To proof the formula for the geodesic equation we need to calculate the adjoint as defined in Sect.~\ref{adjoint}. For vector fields 
$(X,a)$ and $(Y,b)$ the Lie bracket is given by
$$[(X,a),(Y,b)]=\left(-[X,Y], -(\L_Xb-\L_Ya)\right).$$
We calculate 
\begin{align*}
\big\langle \on{ad}&((X,a))^*G((Y,b)),(Z,c)\big\rangle\\& = G((Y,b),\on{ad}((X,a))(Z,c)) \\&= G\Big((Y,b),\left(-[X,Z], -\L_Xc+\L_Za \right)\Big)
\\&
=\int_{\mathbb R} Y'(x)\big(X'(x)Z(x)-X(x)Z'(x)\big)' +b(x)(-\L_Xc(x)+\L_Za(x))\,dx\\&
=\int_{\mathbb R} Y'(x)\big(X''(x)Z(x) + X'(x)Z'(x)-X'(x)Z'(x)-X(x)Z''(x)\big)\,dx
\\&\qquad +\int_{\mathbb R} b(x)(-c'(x)X(x)+a'(x)Z(x))\,dx\\&
=\int_{\mathbb R}    Z(x) \big(X''(x)Y(x) + (X(x)Y'(x))''+ b(x)a'(x)\big)\,dx
\\&\qquad +\int_{\mathbb R} c(x)(b'(x)X(x)+b(x)X'(x))\,dx\\&
=\Big\langle \big(X''Y + (XY')''+ ba',b'X+bX'\big),(Z,c)\Big\rangle.
\end{align*}
Therefore, the adjoint is given by 
\begin{align*}
\on{ad}(X,a)^* G(Y,b)(x)&=\big(X''Y + (XY')''+ ba',b'X+bX'\big)
 .
\end{align*}
Note that for general $(X,a),(Y,b) \in\A\times\A$ the  adjoint is not an 
element of $G(\A\times\A)$; the same statement is true for $\AA\x \A$. 
But for $(X,a)$ equal to $(Y,b)$  we can rewrite the precedeing equation in a form similar to that in Sect.~\ref{H1metric} to obtain
\begin{align*}
\on{ad}(X,a)^* G(X,a)(x)&= \left(\frac12(X'^2)' + \frac12(X^2)'''+ \frac12(a^2)',a'X+aX'\right)\\
&= \left(\frac12\left(\int^x_{-\infty} X'^2+a^2 dx + (X^2)'\right)'' ,a'X+aX'\right)\\
&= \check G \left(\frac12\left(\int^x_{-\infty} X'^2+a^2 dx + (X^2)'\right) ,a'X+aX'\right).
\end{align*}
An argumentat similar to that in Sect.~\ref{H1metric} proves the existence of the adjoint and, thus, of the geodesic equation on $\wt\M$ and the non-existence on $\M$.
\end{proof}

\begin{thm*}[$R$-map for 2HS equation]
Define the map
$$ R:\left\{
\begin{aligned}
 \wt M&\to 
\big(\A(\R,\mathbb C/\{-2\}), \|\cdot\|_{L^2}\big)\\
(\ph,\al) &\mapsto
2\;\ph'^{1/2} e^{i\al/2}-2.
\end{aligned}\right.
$$
The $R$-map is invertible with inverse
$$R\i :\left\{
\begin{aligned}\A(\R,\mathbb C/\{-2\}) &\to \wt M
\\\ga&\mapsto \Big(x+\frac14 \int_{-\infty}^x (|\ga+2|^2-4)\;dx, 2\on{arg}(\ga(x)+2)\Big) .
\end{aligned}\right.$$ 
The pullback of the flat $L^2$-metric via $R$ is the  metric $G$ as defined in the previous theorem. 
Thus the space $\big(\wt M,G\big)$ is a flat space in the sense of Riemannian geometry.
\end{thm*}
\begin{proof}
An argument similar to that in Sect.~\ref{RmapDiffAA} shows that the image $R(\ph,\al)$ is an element of $\A(\R,\mathbb C/\{-2\})$.
The bijectivity follows from a straightforward calculation using that for $\ga=R(\ph,\al)=R(\on{Id}+f,\al)$ we have 
$$\frac14|\ga(x)+2|^2-1=f'(x) ,$$ which proves the identities $R\circ R\i =R\circ R\i=\on{Id}$.

Since the mapping $R$ is bijective, the pullback via $R$ yields a well-defined metric on $\wt\M$. To obtain its formula, we must calculate the tangent mapping of $R$.
Let $(h,U)=(X\circ\ph,U) \in T_{\ph,\al}\wt\M$.
We have
\begin{align*}
T_{\ph,\al} R (h,U) &=  \ph_x^{-1/2} h' e^{i\al/2}+i \ph_x^{1/2} e^{i\al/2}U   .
\end{align*}
Using this formula we have for $h=X_1\circ\ph,k= X_2\circ\ph$: 
\begin{align*}
R^*\langle (h,U) ,(k,V) \rangle_{L^2} &= \on{Re} \int_{\R} \langle T_{\ph,\al} R (h,U) , \overline{T_{\ph,\al} R (k,V)}\rangle dx\\&= \int_\R X_1'(x)X_2'(x) + \al(x)\be(x) dx\\& =  G_{\ph,\al}((h,U) ,(k,V)) .\quad\qedhere
\end{align*}
\end{proof}
We can now again use this result to construct an analytic solution formula for the corresponding geodesic equation --- the two-component HS equation.

\begin{thm*}[Solutions to  $2$HS -equation]
Given an initial value $(u_0,\rho_0)$ in $\AA(\R)\times\A(\R)$  the solution to the HS equation is given by
\begin{align*}
(u,\rho)=(\ph_t\circ\ph\i,-\al\circ\ph+\al_t\circ\ph\i)\quad\text{with}\quad
 (\ph,\al)=R\i(t(u_0'+i\rho_0)) .
\end{align*}
In particular, this means that a solution breaks in finite time $T$, if and only if there exists a $x\in \R$ such that $u'_0(x)<0$ and $\rh_0(x)=0$.
\end{thm*}
\begin{proof}
By the previous theorem we know that the path 
$$
(\ph(t,x),\rho(t,x))=R\i(t\; \ga_0)(x)
$$ 
is a solution to the geodesic equation for every $\ga_0 \in\AA(\R,\mathbb C/\{-2\})$. 
It remains to choose $\ga$ such that the initial conditions
are satisfied. This can be achieved exactly by choosing $\ga_0=T_{\on{Id,0}}R (u_0,\rho_0) = (u_0'+i\rho_0)$.
\end{proof}

\subsection*{Remark} This theorem holds also in more general situations, i.e., for the spaces 
$\AA(\mathbb R)\x \mathcal C(\mathbb R)$ with $\A\neq\mathcal C$, e.g., 
$W^{\infty,1}_1(\mathbb R)\x \mathcal S(\mathbb R)$. The result holds in this situation since 
the diffeomorphism group $\Diff_{\AA}(\R)$ acts on $\mathcal C(\R)$ for all choices of
$\A$ and $\mathcal C$ among $C^{\infty}_c(\R)$, $\mathcal S(\R)$ and $W^{\infty,1}(\R)$. 

\section{Remarks on Periodic Case}\label{periodic}

In this section we will briefly review the results of \cite{Lenells2007} and extend them to the case of real-analytic or ultra-differentiable functions on the circle.
In this section, $\Diff^{\circ}(S^1)/S^1$ denotes one of the following homogeneous spaces:
\begin{enumerate}
 \item $\Diff(S^1)/S^1$ the space of smooth  diffeomorphisms on the circle modulo rotations.
 \item $\Diff^{\om}(S^1)/S^1$ the space of real-analytic diffeomorphisms on the circle modulo 
        rotations; c.f. Sect~\ref{diffom}.
 \item $\Diff^{[M]}(S^1)/S^1$ the space of ultra-differentiable diffeomorphisms of  Beurling type  or 
        Roumieu type on the circle modulo rotations; c.f. Sect.~\ref{diffDC}.
\end{enumerate}
A diffeomorphism  $\ph\in \Diff^{\circ}(S^1)$ is related to its universal covering 
diffeomorphism $\wt\ph$ by $\ph(e^{ix}) = e^{i\wt\ph(x)}$. Then $\wt\ph(x)=x+f(x)$ where 
$f$ is a $2\pi$-periodic real-valued function. Rotations correspond to constant functions $f$.
Let $\widetilde{\Diff}^\circ(S^1)$ denote the regular Lie group of lift to the universal cover of 
diffeomorphisms. The corresponding homogeneous space is then 
$\widetilde{\Diff}^\circ(S^1)/\mathbb R$, factoring out all translations.

\begin{thm*}[\cite{Lenells2007}]
On the homogeneous space  $\widetilde{\Diff}^{\circ}(S^1)/\mathbb R$ the square root representation is a bijective 
mapping given as follows:  
$$
  R:\left\{\begin{aligned}
            \quad\widetilde{\Diff}^{\circ}(S^1)/\mathbb R &\to \left(\on{Im}(R),\|\cdot\|_{L^2([0,2\pi])}\right)\subset
\left(C^{\circ}_{2\pi\text{-per}}(\mathbb R,\mathbb R_{>0}),\|\cdot\|_{L^2}\right)\\
\wt\ph &\mapsto 2\;(\wt\ph')^{1/2} .
           \end{aligned}
\right.
$$
The image of the $R$-map is the sphere of radius $\sqrt{8\pi}$, i.e.,
$$\on{Im}(R) = 
\Big\{\ga \in C^{\circ}_{2\pi\text{-per}}(\mathbb R,\mathbb R_{>0}):
\|\ga\|^2_{L^2}=\int_{0}^{2\pi} \ga^2\,d\th = 8\pi\Big\}.$$
The pullback of the  restriction of the $L^2$-metric to $\on{Im}(R)$ via $R$ 
is the homogeneous Sobolev metric of order one, i.e.,
$$R^*\langle \cdot,\cdot\rangle = \langle\cdot,\cdot\rangle_{\dot H^1} .$$
Thus the spaces $\big(\Diff^{\circ}(S^1)/S^1,\dot H^1\big)$ have constant positive sectional curvature.
\end{thm*}
Here $C^{\circ}_{2\pi-\text{per}}(\mathbb R,\mathbb R_{>0})$ 
denotes the corresponding space of sufficiently smooth functions, i.e., 
either the space of $C^{\infty}$, real-analytic or ultra-differentiable functions.
\begin{proof}
The case of $C^{\infty}$-functions is proven in the work of Lenells \cite{Lenells2007}.
The statement about the image of $R$ follows from
$$\|R(\wt\ph)\|^2 = 4\|\;(\wt\ph')^{1/2} \|^2 = 4\int_0^{2\pi} \wt\ph'(x) dx= 4\int_0^{2\pi}1+f'(x) dx =8\pi$$
The remaining cases follow similarly using that 
$\Diff^{\om}(S^1)$ and $\Diff^{[M]}(S^1)$ are Lie subgroups of $\Diff(S^1)$, c.f.~\ref{diffom}
and~\ref{diffDC}.
\end{proof}
As a direct consequence we obtain the following result:
\begin{thm*}[Solutions to the periodic HS equation]
Given an initial value $u_0$ in $C^{\circ}(S^1,\mathbb R_{>0})$  the solution to the HS equation stays locally in the same space.
A solution exists  for all time $t$, if and only if $u'_0(\th)\geq0$ for  all $\th\in S^1$.
\end{thm*}
\subsection*{Remark}
From our setup in Sects. \ref{diffom} and \ref{diffDC} it is obvious that 
the results of \cite{Lenells2011_preprint} for the  two-component HS equation extend 
to the cases of real-analytic and ultra-differentiable functions.

\section{Similar Representation for Camassa-Holm Equation}

In this article we have shown that certain non-trivial Riemannian spaces that have flat or constant curvature can be represented 
as a simple submanifold of the flat manifold of all sufficiently smooth functions equipped with the $L^2$-metric. In this section  we will present a natural example of a metric space
with non-trivial curvature which can also be represented as a (complicated) subspace  of the flat manifold of all  smooth functions, namely the Lie group $\Diff(S^1)$ equipped with 
the right-invariant non-homogeneous $H^1$-metric.

\begin{thm*}[\cite{Kouranbaeva1999}]
The right-invariant $H^1$-metric on the Lie group $\Diff(S^1)$ is given by
$$G_{\ph}(X\circ\ph,Y\circ\ph) = \int X(x)Y(x) + X'(x)Y'(x)\;dx,$$
where
$X,Y$ are vector fields in the  Lie algebra $\X(S^1)$. 
The induced geodesic distance is positive, and
the corresponding geodesic equation is the Camassa-Holm Eq. \cite{Holm1993} given by
$$\boxed{\begin{aligned}
u_t- u_{xxt} + 3 u u_x = 2 u_x u_{xx} + u u_{xxx}.
\end{aligned}}$$
The geodesic equation is well posed and the exponential map is a local diffeomorphism.
\end{thm*}
Using the ideas of \cite{Lenells2011_preprint} we can introduce an $R$-map for this space.
We use again $\ph(e^{ix}) = e^{i\wt\ph(x)}$ with $\wt\ph(x)=x+f(x)$ for periodic $f$.
Again let $\widetilde{\Diff}(S^1)$ denote the regular Lie group of lifts to the universal cover of 
diffeomorphisms. 
For curves we obtain $\p_t|_0 \ph(t,e^{ix}) = i\p_t|_0\wt\ph(t,x).e^{i\wt\ph(x)} = i 
\p_t|_0 f(t,x).e^{i\wt\ph(x)}$. Thus, for tangent vectors we obtain 
$\de\ph = i.\de\wt\ph.\ph = i.\de f.\ph$. 

\begin{thm*}
The $R$-map is defined by 
\begin{align*}
&R(\wt\ph) :=2\wt\ph'^{\frac12}e^{i (\wt\ph-\on{Id})/2} = 2(1+f')^{\frac12}e^{i f/2},
\\ 
&R:\widetilde{\Diff}(S^1)\to C^{\infty}_{2\pi\text{-per}}(\mathbb R,\mathbb C) .
\end{align*}
The image under the $R$-map of the diffeomorphism group is the space $\mathcal S$ given by 
\begin{align*}
\mathcal S:= R(\widetilde{\Diff}(S^1)) &= 
\Big\{\ga \in C^{\infty}_{2\pi\text{-per}}(\mathbb R,\mathbb C \setminus\{0\}): F(\ga) = (F_1(\ga),F_2(\ga))= 0\Big\},
\\
\text{where  }F_1(\ga):&=\int_0^{2\pi} (|\ga|^2-1)\,d\th,\\
\text{and } F_2(\ga):&=8\on{arg}(\ga)'-|\ga|^2.
\end{align*}
The $R$ map is injective,  and 
for any curve in   $\mathcal S$ the inverse of $R$ is given by 
$$
R\i(\ga) = 2\on{arg}(\ga)+ \on{Id}_{\mathbb R} \qquad R\i:\mathcal S\mapsto \Diff(S^1).$$
Furthermore, the pullback of the $L^2$-inner product on $C^{\infty}(S^1,\mathbb C)$ 
to the  diffeomorphism group by the $R$-map is the right-invariant  Sobolev metric of order one.
\end{thm*}

\begin{proof}
To prove the characterization for the image of $R$ we observe that for $\ga\in C^{\infty}(S^1,\mathbb C \setminus\{0\})$ 
the function $R\i(\ga)\in C^{\infty}([0.2\pi),[0,2\pi))$ is periodic if and only if $F_1(\ga)=0$. Furthermore, 
we have that $R(R\i(\ga))=\ga$ if and only if $F_2(\ga)=0$.
 
To calculate the formula for the pullback metric, we need to calculate the tangent of the $R$-map 
where $h$ is tangent to $\wt\ph$, i.e., to $f$.
\begin{align*}
T_{\wt\ph}Rh &= \wt\ph'^{-\frac12}h'e^{if/2}+i \wt\ph'^{\frac12}e^{if/2}h.
\end{align*}
Thus the pullback of the $L^2$ inner product on $C^{\infty}(S^1,\mathbb C)$ is given by
\begin{align*}
(R^*\langle ~,~ \rangle_{L^2})_{\wt\ph}(h,h) &= \int_{0}^{2\pi}T_{\wt\ph}R(h)\cdot 
\overline{T_{\wt\ph}R(h)}\, dx\\
&= \int_{0}^{2\pi} \frac{h'^2}{ \wt\ph'} + h\wt\ph'\, dx=\int_{S^1} X^2+X'^2\,dx,
\end{align*}
with $h=X\circ\wt\ph$.
\end{proof}

\bibliographystyle{plain}

\end{document}